\documentclass[11pt,reqno]{amsart}

\usepackage{amsmath}
\usepackage{amsthm}
\usepackage{mathrsfs}
\usepackage{graphicx}
\usepackage{mathtools}
\usepackage{amsfonts}
\usepackage{amssymb}
\usepackage{hyperref}
\usepackage[utf8]{inputenc}
\usepackage[margin=1 in]{geometry}
\usepackage{enumerate}
\usepackage[normalem]{ulem}
\usepackage{tikz}
\usepackage{tikz-cd}
\usepackage{xcolor}
\usetikzlibrary{matrix,arrows,positioning,automata}
  \usepackage{cancel}
\usepackage{dsfont}

  \usepackage{scalerel,stackengine}
\stackMath
\newcommand\reallywidehat[1]{%
\savestack{\tmpbox}{\stretchto{%
  \scaleto{%
    \scalerel*[\widthof{\ensuremath{#1}}]{\kern-.6pt\bigwedge\kern-.6pt}%
    {\rule[-\textheight/2]{1ex}{\textheight}}%WIDTH-LIMITED BIG WEDGE
  }{\textheight}% 
}{0.5ex}}%
\stackon[1pt]{#1}{\tmpbox}%
}

\newcommand{\mcl}[1]{{\mathcal #1}}

\newcommand{\I}{\mathrm I}
\newcommand{\II}{\mathrm{II}}
\newcommand{\III}{\mathrm{III}}

\newcommand{\ip}[1]{\left\langle #1 \right\rangle}

\numberwithin{equation}{section}
\newcommand{\N}{\mathbb N}
\newcommand{\R}{\mathbb R}
\def\E{\mathbb E}

\newcommand{\linf}{L^{\infty}}
\newcommand{\sF}{\mathcal{F}}
\newcommand{\bP}{\mathbb{P}}

\newcommand{\sL}{\mathcal{L}}

\newcommand{\bx}{\bm{x}}
\newcommand{\bX}{\boldsymbol X}

\def\XXint#1#2#3{{\setbox0=\hbox{$#1{#2#3}{\int}$}
\vcenter{\hbox{$#2#3$}}\kern-.5\wd0}}

\newcommand{\T}{\mathbb{T}}

\numberwithin{equation}{section}
\newtheorem{thm}{Theorem}[section]
\newtheorem{lem}[thm]{Lemma}
\newtheorem{cor}[thm]{Corollary}
\newtheorem{prop}[thm]{Proposition}
\newtheorem{assumption}[thm]{Assumption}
\theoremstyle{definition}
\newtheorem{defn}[thm]{Definition}
\newtheorem{rmk}[thm]{Remark}

\def\smallnegint{\mathop{\int\mkern-13mu
        \raise.5ex\hbox{${\scriptscriptstyle\diagup}$}}\nolimits}
\def\ds{\displaystyle}

\def\div{\operatorname{div}}
\def\tr{\operatorname{tr}}

\def\bx{{\boldsymbol x}}
\def\by{{\boldsymbol y}}
\def\ssetminus{\,\raise.4ex\hbox{$\scriptstyle\setminus$}\,}

\newcommand{\be}{\begin{equation}}
\newcommand{\ee}{\end{equation}}

\newcommand{\bc}{\begin{case}}
\newcommand{\ec}{\end{cases}}
\newcommand{\bs}{\begin{split}}
\newcommand{\es}{\end{split}}

\newcommand{\norm}[1]{\left\Vert#1\right\Vert}
\newcommand{\abs}[1]{\left\vert#1\right\vert}

\newcommand{\bm}[1]{\boldsymbol #1}

\renewcommand{\bar}{\overline}

\renewcommand{\hat}{\widehat}

\def \be {\begin{equation}}
\def \ee {\end{equation}}

\def \E {\mathbb{E}}

\def \R {\mathbb{R}}

\newcommand{\pr}{\mathcal{P}}
\newcommand{\cC}{\mathcal C}

\newcommand{\bbF}{\mathbb{F}}

\newcommand{\trip}[1]{{\left\vert\kern-0.25ex\left\vert\kern-0.25ex\left\vert #1 
    \right\vert\kern-0.25ex\right\vert\kern-0.25ex\right\vert}}

\newcommand{\cP}{\mathcal{P}}

\newcommand{\id}{\text{Id}}
\newcommand{\ov}{\overline}
\newcommand{\bd}{\mathbf{d}}
\newcommand{\cA}{\mathcal{A}}
\newcommand{\cB}{\mathcal{B}}
\newcommand{\cpart}{\mathcal{C}^{1,2,2}_{\text{p}}}

\title[Well-posedness of Hamilton-Jacobi equations in the Wasserstein space]{Well-posedness of Hamilton-Jacobi equations in the Wasserstein Space: non-convex Hamiltonians and common noise}

\begin{document}

\author[S. Daudin]{Samuel Daudin 
\address{(S. Daudin) Universit\'e C\^ote d'Azur, CNRS, Laboratoire J.A. Dieudonn\'e, 06108 Nice, France
}\email{samuel.daudin@univ-cotedazur.fr
}}

\author[J. Jackson]{Joe Jackson\address{(J. Jackson) The University of Chicago, Eckhart Hall, 5734 S University Ave, Chicago, IL 60637, USA}
\email{jsjackson@uchicago.edu}}
 
\author[B. Seeger]{ Benjamin Seeger
\address{(B. Seeger) University of Texas at Austin, 2515 Speedway, PMA 8.100, Austin, TX 78712, USA
}\email{seeger@math.utexas.edu
}}

\thanks{S. Daudin acknowledges the financial support of the European Research Council (ERC) under the European Union's Horizon 2020 research and innovation program (ELISA project, Grant agreement No. 101054746). J. Jackson is supported by the NSF under Grant No. DMS2302703. B. Seeger was partially supported by the National Science Foundation (NSF) under award numbers DMS-1840314 and DMS-2307610.}

\maketitle

\begin{abstract}
    We establish the well-posedness of viscosity solutions for a class of semi-linear Hamilton-Jacobi equations set on the space of probability measures on the torus. In particular, we focus on equations with both common and idiosyncratic noise, and with Hamiltonians which are not necessarily convex in the momentum variable. Our main results show (i) existence, (ii) the comparison principle (and hence uniqueness), and (iii) the convergence of finite-dimensional approximations for such equations. Our proof strategy for the comparison principle is to first use a mix of existing techniques (especially a change of variables inspired by \cite{bayraktar2023} to deal with the common noise) to prove a ``partial comparison" result, i.e. a comparison principle which holds when either the subsolution or the supersolution is Lipschitz with respect to a certain very weak metric. Our main innovation is then to develop a strategy for removing this regularity assumption. In particular, we use some delicate estimates for a sequence of finite-dimensional PDEs to show that under certain conditions there \textit{exists} a viscosity solution which is Lipschitz with respect to the relevant weak metric. We then use this existence result together with a mollification procedure to establish a full comparison principle, i.e. a comparison principle which holds even when the subsolution under consideration is just upper semi-continuous (with respect to the weak topology) and the supersolution is just lower semi-continuous. We then apply these results to mean-field control problems with common noise and zero-sum games over the Wasserstein space.
\end{abstract}

\setcounter{tocdepth}{1}
\tableofcontents

\section{Introduction}

This paper is concerned with the Hamilton-Jacobi equation 
\begin{align} \label{hjb} \tag{\text{HJB}}
\begin{cases}
    \ds - \partial_t U - (1 + a_0) \int_{\T^d} \tr(D_x D_m U\big)(t,m,x)  m(dx) - a_0 \int_{\T^d} \int_{\T^d} \tr\big(D_{mm} U\big)(t,m,x,y) m(dx) m(dy) \vspace{.1cm} \\ \ds
     \hspace{2cm} + \int_{\T^d} H\big(x, D_m U(t,m,x), m \big) m(dx) = 0, \quad (t,m) \in [0,T] \times \cP(\T^d), \vspace{.1cm} \\
    U(T,m) = G(m), \quad m \in \cP(\T^d).
    \end{cases}
\end{align}
The data consists of two functions 
\begin{align*}
    H = H(x,p,m) : \T^d \times \R^d \times \cP(\T^d) \to \R, \qquad G = G(m) : \cP(\T^d) \to \R, 
\end{align*}
as well as the parameter $a_0 \geq 0$ which represents the intensity of the common noise. The unknown is a map $U = U(t,m) : [0,T] \times \cP(\T^d) \to \R$. 

Equations of the form \eqref{hjb} have received significant attention in recent years, in large part because of their connection to mean field control. In particular, when $H$ is convex in $p$, \eqref{hjb} is the dynamic programming equation of a corresponding mean field control problem. When $H$ is non-convex (i.e. not convex in $p$), \eqref{hjb} arises in the study of zero-sum stochastic differential games of mean field type, as proposed in \cite{cossophamjmpa}. We refer also to \cite{mourrat} for an equation similar to \eqref{hjb} which arises in the setting of spin glasses, and where the non-convexity of the Hamiltonian creates some interesting challenges.

It turns out that \eqref{hjb} may not have a classical solution even if the data is smooth. In the convex case this is related to the non-uniqueness of optimal trajectories for the control problem, and an example can be found in \cite{Briani2018}. As a consequence, it is necessary to work with weak solutions, and various notions of viscosity solutions have been proposed and studied. The comparison principle for viscosity sub/super-solutions of \eqref{hjb} is particularly subtle, and has received a huge amount of attention in recent years. We refer to \cite{GT_19, confortihj, Cosso2021, Soner2022, Soner2023, daudinseeger, bayraktar2023, shao2023, talbitouzizhang, bertucci2023} for recent efforts in this direction. These references employ various notions of viscosity solutions, and work under a variety of assumptions (convex or non-convex $H$, common noise or no common noise, degenerate or non-degenerate idiosyncratic noise, semilinear or fully nonlinear) and so we make no attempt to give a full description of the state of the art. Instead, we describe in the next subsection the approaches which are most relevant to the present paper.

In addition to the comparison principle, we are interested in the connection between \eqref{hjb} and the sequence of finite-dimensional PDEs 
\begin{align} \label{hjbn} \tag{$\text{HJB}_N$}
\begin{cases}
    \ds - \partial_t V^N - \sum_{i = 1}^N \Delta_{x^i} V^N - a_0 \sum_{i,j = 1}^N \tr\big(D_{x^ix^j} V^N \big) + \frac{1}{N} \sum_{i = 1}^N H(x^i, ND_{x^i} V^N, m_{\bx}^N) = 0, \vspace{.1cm} \\
    \hspace{8cm} (t,\bx) \in [0,T] \times (\T^d)^N, \vspace{.1cm} \\
    V^N(T,\bx) = G(m_{\bx}^N), \quad \bx \in (\T^d)^N.
    \end{cases}
\end{align}
In particular, it is expected that the viscosity solution $U$ of \eqref{hjb} arises as the limit of the solutions $V^N$ to \eqref{hjbn}, in the sense that 
\begin{align} \label{convintro}
    \lim_{N \to \infty} \sup_{(t,\bx) \in [0,T] \times (\T^d)^N} |V^N(t,\bx) - U(t,m_{\bx}^N)| = 0.
\end{align}
When $H$ is convex, \eqref{convintro} is related to the convergence problem in mean field control, and can be verified by control-theoretic compactness arguments (see \cite{Lacker2017} and \cite{djete2022mor}). In the convex setting, it is in fact possible to sharpen the asymptotic statement \eqref{convintro} with a rate of convergence, see \cite{CDJS_23,DDJ_23}. When $H$ is not convex, the convergence \eqref{convintro} is not as well studied, but has been established in the case of purely common noise in \cite{GangboMayorgaSwiech}.

As indicated in the title, the purpose of this paper is to address well-posedness for \eqref{hjb} (in particular the comparison principle) and convergence for \eqref{hjbn} in the presence of both idiosyncratic and common noise ($a_0 > 0$) and without assuming that $H$ is convex in $p$. We briefly summarize our main results here (precise assumptions and results will be state in Section \ref{sec:prelim} below):
\begin{itemize}
\item We prove a comparison principle (Theorem \ref{thm.comparison}), and hence uniqueness, for viscosity solutions of \eqref{hjb} for general (nonconvex) $H$ and under minimal growth and regularity assumptions the data and the solutions. In particular, Lipschitz regularity is measured with respect to the $\mathbf{d}_1$-metric, which generalizes results of \cite{Soner2022,bayraktar2023} (see the discussion below).
%\textcolor{red}{It is not clear to what function this regularity is related to}
\item Relying purely on PDE arguments, we establish the convergence statement \eqref{convintro} (Theorem \ref{thm.convergence}). We thus extend the convergence results of \cite{Lacker2017, djete2022mor, GangboMayorgaSwiech} to equations with both common noise and idiosyncratic noise and with possibly non-convex Hamiltonians. In order to prove this convergence result, we verify that every ``limit point" of the sequence $(V^N)_{N \in \N}$ (given by \eqref{hjbn}) is a viscosity solution of \eqref{hjb}. Together with a compactness argument and the comparison principle, this implies that there exists a unique viscosity solution of \eqref{hjb} (Theorem \ref{thm.exist}), and it arises as the limit of the solutions of the finite-dimensional equations \eqref{hjbn}. 
\item The equation \eqref{hjb} is shown in Theorem \ref{thm.regularity} to propagate regularity. Namely, if $H$ and $G$ are sufficiently smooth, then the unique viscosity solution $U$ is Lipschitz continuous in $\cP(\T^d)$ with respect to the $C^{-k}$-metric, where $k \in \N$ depends on the smoothness of $H$ and $G$.
\item Finally, we use the equations \eqref{hjb} and \eqref{hjbn} to study mean field control and mean field zero-sum games (depending on whether or not $H$ is convex), as well as the analogous large population models. Among the new results we prove are a) the characterization of the value function from mean field control with idiosyncratic and common noise as the unique viscosity solution of \eqref{hjb}, b) the convergence \eqref{convintro} for the upper and lower value functions for $N$-player  stochastic differential games to the corresponding value functions for the mean field problems, and c) the coincidence of the upper and lower mean-field value functions under the so-called Isaacs condition. 
\end{itemize}

\subsection{Some existing approaches to comparison} 

In order to better explain the difficulties we face, we now summarize a number of existing approaches to the comparison principle, and explain why they do not seem to apply to equations with non-convex Hamiltonians and/or both common and idiosyncratic noise.
\newline
\paragraph{\textit{Exploiting the connection to mean field control:}} When $H$ is convex in $p$, the connection to mean field control makes it relatively easy to establish the existence of a solution - one checks by dynamic programming that the value function of the control problem is a viscosity solution of the equation. This connection can in fact also be helpful in establishing a comparison principle. For example, when $a_0 = 0$ and $H$ is convex in $p$, the equation \eqref{hjb} is covered by the results in \cite{Cosso2021}. While in some ways the comparison principle in \cite{Cosso2021} is remarkably general (it covers a class of fully non-linear equations), the approach in \cite{Cosso2021} relies heavily on the connection to mean field control. In particular, in this setting the convergence \eqref{convintro} has already been established by control-theoretic arguments in \cite{Lacker2017}, and the proof of comparison in \cite{Cosso2021} uses this convergence result heavily. Thus while it is possible that the techniques of \cite{Cosso2021} could be adapted to the case $a_0 > 0$, it does not seem that they can be used when the Hamiltonian $H$ is not convex in $p$.
\newline 
\paragraph{\textit{The lifting approach:}} 
For equations with no idiosyncratic noise (roughly speaking, this means that the $(1+ a_0)$ appearing in \eqref{hjb} is replaced with $a_0$), it is possible to define a notion of viscosity (sub/super)-solutions by ``lifting" the equation to a Hilbert space of square-integrable random variables as in \cite{GT_19}. An extensive survey of optimal control problems in infinite-dimensional Hilbert spaces can be found in the book \cite{FGS_book}. See also \cite{GangboMayorgaSwiech} and \cite{mayorgaswiech}, where comparison principles obtained through this lifting approach are used to show convergence. The lifting approach does not rely on the control formulation, and so works for non-convex Hamiltonians. On the other hand, there is a major obstacle to adapting this approach to the case of non-zero idiosyncratic noise - roughly speaking, the ``common noise operator" 
\begin{align} \label{commonnoiseop}
    a_0 \Big( \int_{\T^d} \tr(D_x D_m U)(t,m,x) m(dx) + \int_{\T^d} \int_{\T^d} \tr(D_{mm} U)(t,m,x,y)m(dx)m(dy) \Big)
\end{align}
appearing in \eqref{hjb} ``lifts" nicely to the Hilbert space, but the ``idiosyncratic noise operator"
\begin{align} \label{idionoiseop}
    \int_{\T^d} \tr(D_x D_m U)(t,m,x) m(dx)
\end{align}
results in a singular second-order operator. For a short, self-contained explanation of this difficulty, see the introduction of \cite{daudinseeger}. Thus the lifting approach does not seem immediately applicable to the equation \eqref{hjb}.
\newline 
\paragraph{\textit{Doubling of variables with entropy penalization:}} One of the basic roadblocks when trying to prove the comparison principle for \eqref{hjb} is that many of the natural metrics on $\cP(\T^d)$ (for example the 1-Wasserstein metric $\bd_1$ or the 2-Wasserstein metric $\bd_2$) do not have good regularity properties, and in particular maps like $m \mapsto \bd_2^2(m,m_0)$ for fixed $m_0 \in \cP(\T^d)$ are not typically smooth. This means that these metrics cannot be used naively as a penalization in a standard doubling of variables argument. The recent work \cite{daudinseeger} of the first and last named authors circumvents this issue by adding an extra ``entropy penalization" in the doubling of variables argument, to force the maximum in the relevant optimization problem to occur at a place where the squared 2-Wasserstein distance is in fact differentiable in an appropriate sense; see also works of Feng et al. where similar techniques were used in a slightly different context \cite{Feng_Kats_2009, Feng_Mikami_Zimmer_2021, Feng_Swiech_2013, Feng2012}. This approach ultimately led to a comparison principle for an equation of the form \eqref{hjb} with nonconvex $H$ but with no common noise ($a_0 = 0$), but it seems that there are serious challenges to adapting the approach to the case $a_0 > 0$. Moreover, we mention that while the comparison principle in \cite{daudinseeger} is established without convexity of $H$, the notion of viscosity solution employed there is more strict than the one we use here, and this makes existence more challenging. In particular, existence is established in \cite{daudinseeger} only when $H$ is convex in $p$, by appealing to the control formulation.
\newline 
\paragraph{\textit{Doubling of variables with weaker metrics:}} Another way to get around the regularity issues in the standard doubling of variables argument is to simply work with a metric $\rho$ whose square is sufficiently smooth. Heuristically, it seems that \textit{weaker} metrics on $\cP(\T^d)$ have better regularity properties, and in particular \cite{Soner2022} showed that a comparison result can be obtained by doubling variables and penalizing with the metric inherited from a ``sufficiently negative" Sobolev space. This idea was later combined with a clever change of variables in \cite{bayraktar2023} to cover the case $a_0 > 0$. In particular, \cite{bayraktar2023} took advantage of the fact that the ``common noise operator" appearing in \eqref{hjb} can be re-cast as a finite-dimensional Laplacian. This allowed \cite{bayraktar2023} to obtain the first comparison principle for an equation of the form \eqref{hjb} (in fact, a much more general version of \eqref{hjb}) with both common and idiosyncratic noise. Their strategy was to double variables and penalize with a weak metric, as in \cite{Soner2022}, and then handle the additional challenge of the common noise operator with a finite-dimensional Ishii's lemma. 

The drawback of working with weaker metrics is that, in order to execute the doubling of variables argument, it seems necessary to assume that either the subsolution or the supersolution under consideration is uniformly Lipschitz with respect to this weak metric (see the hypotheses in Theorem 4.1 of \cite{Soner2022} and Theorem 3.2 of \cite{bayraktar2023}). In general it is not a priori clear when there exists a solution satisfying this Lipschitz regularity, which makes these comparison principles difficult to apply. Even in the case of \eqref{hjb} with $H$ convex in $p$, the value function of the associated mean field control problem is not expected to be uniformly Lipschitz with respect to the relevant weak metric unless the data is very smooth and the noise is non-degenerate (see \cite{DDJ_23} for some results along these lines in the case $a_0 = 0$). It is therefore not clear whether the comparison principles in \cite{bayraktar2023, Soner2022} can be applied apart from these special situations.

%\red{Maybe we can mention somewhere that $H^{-k}$ Lipschitz regularity fails in the first order case and give and example (purely quadratic Hamiltonien and linear terminal cost)  }

%\textcolor{red}{We need to mention \cite{Soner2023} where they succesfully use the doubling of variables with weak metrics for only $d_1$-regular sub/super-solutions. It only works for very restrictive classes of Hamiltonians  but uses in a clever way the uniform ellipticity in the doubling of variables argument, a little bit like our entropy penalization with Ben. Edit: I'm now convinced that their comparison proof is false...}

\subsection{Our results} 

To explain our main results, we first need to discuss our definition of viscosity (sub/super)solutions. The notion of viscosity solution we adopt is inspired by the proof strategy in \cite{bayraktar2023}, and in particular the observation that if $U$ is smooth and we define a function $\ov{U} = \ov{U}(t,z,m) : [0,T] \times \T^d \times \cP(\T^d) \to \R$ via the formula 
\begin{align} \label{def.ovu}
    \ov{U}(t,z,m) = U\big(t, (\id + z)_{\#} m \big), 
\end{align}
then we have 
\begin{align*}
    &\Delta_z \ov{U}(t,z,m) = \int_{\T^d} \tr(D_x D_m U \big)(t,m^z,x)  m^z(dx) + \int_{\T^d} \int_{\T^d} \tr\big(D_{mm} U\big)(t,m^z,x,y) m^z(dx) m^z(dy),
\end{align*}
where here and throughout the paper we use the notation
\begin{align}
    m^z = (\id + z)_{\#} m \label{def.mz}
\end{align}
for simplicity. Using this observation, one can check that if $U$ is a smooth solution to \eqref{hjb}, then $\ov{U}$ satisfies the PDE 
\begin{align} \label{hjbz} \tag{$\text{HJB}_z$}
\begin{cases}
    \ds - \partial_t \ov{U} - \int_{\T^d} \tr(D_x D_m \ov{U}\big)(t,z,m,x)  m(dx) - a_0 \Delta_z \ov{U}(t,z,m) \vspace{.1cm} \\
    \ds \hspace{2cm} + \int_{\T^d} \ov{H}(x,z,D_m \ov{U}(t,z,m,x),m \big) m(dx) = 0, \quad (t,m) \in [0,T] \times \cP(\T^d), \vspace{.1cm} \\
    \ov{U}(T,z,m) = \ov{G}(z,m), \quad (z,m) \in \T^d \times \cP(\T^d),
    \end{cases}
\end{align}
where we have set 
\begin{align}
   \ov{H}(x,z,p,m) \coloneqq H(x+z, p , m^z), \quad \ov{G}(z,m) \coloneqq G(m^z).
\end{align}
The equation \eqref{hjbz} is much easier to deal with, because the key ``common noise operator," which contains true second-order derivatives in the measure variable, has been exchanged for a finite-dimensional Laplacian. Thus we define viscosity solutions for the original equation \eqref{hjb} in terms of viscosity solutions of the transformed equation \eqref{hjbz}, i.e. we say a function $U$ is a viscosity (sub/super)solution of \eqref{hjb} if the function $\ov{U}$ defined by \eqref{def.ovu} is a viscosity (sub/super)solution of \eqref{hjbz}. Viscosity (sub/super)solutions of \eqref{hjbz}, meanwhile, are defined in a natural way in terms of smooth test functions (see Definition \ref{def.viscosityz} below for the precise definition).

Our main results are (i) existence, (ii) the comparison principle, and (iii) convergence of the finite-dimensional approximations for the equation \eqref{hjb} under the assumption that $G$ is Lipschitz with respect to $\bd_1$, and $H$ satisfies a structural condition like
\begin{align*}
    |H(x,p,m) - H(x',p',m')| \leq C(1 + |p| + |p'|)\big(|x - x'| + |p - p'| + \bd_1(m,m') \big).
\end{align*}
We emphasize two points about our main results. First, our comparison principle allows to compare between a subsolution $U^1$ which is merely upper semi-continuous and a supersolution $U^2$ which is merely lower semi-continuous, unlike \cite{daudinseeger}, \cite{Soner2022}, \cite{bayraktar2023}, which all require some sort of Lipschitz continuity for one or both of the semisolutions. Second, our existence result does not rely on the representation of the solution in terms of a stochastic control problem or game, and instead we construct the solution as a ``limit point" of the solutions of the finite-dimensional PDEs \eqref{hjbn}. While the change of variables we employ in some sense enlarges the set of test functions available (making comparison easier), the fact that we can construct a solution in a straightforward way (without the help of a representation formula) indicates that our set of test functions is not ``too large", i.e. the notion of viscosity solution we employ is not too restrictive to be widely applicable.

In Section \eqref{sec.applications}, we treat two applications. In Subsection \ref{subsec.mfc}, we use our main results to identify the value function of a mean field control problem with common noise as the unique viscosity solution of the corresponding dynamic programming equation. To the best of the authors' knowledge, this is the first such characterization. In the same setting, our convergence result also gives a new analytical proof of the convergence of the values of the finite particle stochastic control problems towards their mean field limit, although we still rely on the dynamic programming principle in \cite{djete2022aop} to check rigorously that the value function is a viscosity solution of the relevant equation. In Subsection \ref{subsec.zerosum}, we treat a zero-sum game on the Wasserstein space (without common noise, for technical reasons). This game was introduced in \cite{cossophamjmpa}, where dynamic programming principles for the upper and lower value functions were established. Here we obtain for the first time a rigorous proof that (i) the upper and lower value functions $U^+$ and $U^-$ arise as the uniform limits of corresponding finite-dimensional upper and lower value functions $(V^{N,+})_{N \in \N}$ and $(V^{N,-})_{N \in \N}$ and (ii) $U^+ = U^-$ when Isaacs' condition holds.

\subsection{Proof strategy}

Our proof of comparison starts by following the strategy of \cite{bayraktar2023} to prove a \textit{partial} comparison principle, i.e. a comparison result under the additional assumption that one of the semi-solutions has some extra regularity. In particular, we prove in Proposition \ref{prop.partialcomparison} a comparison principle assuming that one of the semisolutions is $H^{-k}$-Lipschitz, i.e. Lipschitz with respect to the metric inherited from the negative Sobolev space $H^{-k}$ defined in Section \ref{sec.comparison} for some large $k \in \N$. 

Much of the paper is devoted removing this extra regularity assumption by constructing approximate sub and supersolutions of \eqref{hjb} with $H^{-k}$-Lipschitz regularity. Since we are able to prove that any limit point of the sequence $(V^N)_{N \in \N}$ is a solution of \eqref{hjb}, we thus reduce the problem to proving corresponding regularity estimates for the finite-dimensional equations \eqref{hjbn}. Unfortunately, it does not seem directly possible to prove an estimate of the form
\begin{align} \label{hsproj}
    |V^N(t,\bx) - V^N(t,\by)| \leq C\| m_{\bx}^N - m_{\by}^N \|_{-s} = C \sup_{\|f\|_s \leq 1} \frac{1}{N} \sum_{i = 1}^N \big(f(x^i) - f(y^i) \big)
\end{align}
uniformly in $N$.
%\red{because one has to handle far more interaction terms involving the points $(x^i)_{i=1}^N$ and $(y^i)_{i=1}^N$. ???} 
On the other hand, it is simple enough to establish the estimate 
\begin{align*}
    |D_{x^i} V^N| \leq C/N, 
\end{align*}
and therefore the results of \cite{fournier2015rate} imply that, if $\sup_{(t,\bx)} |V^{N_k}(t,\bx) - U(t,m_{\bx}^{N_k})| \to 0$ for some subsequence $(N_k)_{k \in \N}$, then it follows that 
\begin{align*}
    \hat{V}^{N_k} \to U, \text{ uniformly on } [0,T] \times \cP(\T^d), \quad \text{where }
    \hat{V}^{N}(t,m) \coloneqq \int_{(\T^d)^{N}} V^N(t,\bx) m^{\otimes N}(d\bx). 
\end{align*}
Remarkably, understanding the Lipschitz regularity of $\hat{V}^{N}$ with respect to negative Sobolev spaces is much more tractable than trying to prove an estimate like \eqref{hsproj}. Indeed, explicit computing the linear derivative $\frac{\delta \hat{V}^N}{\delta m}$ reveals that $\hat{V}^N$ is $C^{-k}$-Lipschitz (hence $H^{-k}$-Lipschitz), uniformly in $N$, provided we can find a constant $C$ such that the estimate 
\begin{align} \label{uniformintro}
    \|D_{x^i}^j V^N \|_{\infty} \leq C/N
\end{align}
holds for each $N \in \N$, $i = 1,...,N$, and $1 \leq j \leq k$. 

The estimate \eqref{uniformintro} is proved by differentiating \eqref{hjbn} repeatedly. As a consequence, the constant $C$ in \eqref{uniformintro} depends on higher order derivatives of $H$ and $G$ in the measure variable. On the other hand, constructing the regular sub/supersolutions of \eqref{hjb} involves replacing $H$ and $G$ by smooth sequences $(H_n)$ and $(G_n)$ through a mollification procedure. The resulting constant $C$ in \eqref{uniformintro} thus blows up in $n$, which provides a complication in showing that the $H^{-s}$-Lipschitz functions we obtain from limits of $(V^N)$ are in fact sub/supersolutions up to small factors that converge to zero as $n \to \infty$. We must then carefully factor the estimate \eqref{uniformintro} as
\begin{align} \label{uniformintro:factor}
    \|D_{x^i}^j V^N \|_{\infty} \leq C_0/N + o(1/N),
\end{align}
where now the constant $C_0$ depends only on bounds for the \emph{first} derivatives in the measure variable of $G$ and $H$.

The most technical part of the paper is Section \ref{sec.uniformest}, where we verify that \eqref{uniformintro:factor} holds provided that the $H$ and $G$ are sufficiently smooth. The strategy for proving \eqref{uniformintro:factor} is reminiscent of the nonlinear adjoint method. In particular, we repeatedly differentiate \eqref{hjbn}, and integrate the resulting equation against an auxiliary Fokker-Planck equation. The $j^{\text{th}}$ iteration of this procedure yields pointwise bounds like \eqref{uniformintro:factor}, but also $L^2$ bounds on quantities like $D D_{x^i}^j V^N$, when integrated against the auxiliary Fokker-Planck equation. Both the $L^{\infty}$ and $L^2$ bounds obtained on the $j^{\text{th}}$ iteration are ``plugged in" to the $(j+1)^{\text{th}}$ iteration, and in the end this scheme allows to prove \eqref{uniformintro:factor} for all relevant $k$ through induction.

To complete the proof of the comparison principle, we start with a viscosity subsolution $U^1$ and a viscosity supersolution $U^2$, and assume without loss of generality that both are strict semi-solutions. Then we mollify the data, to produce sequences $(H^n)_{n \in \N}$, $(G^n)_{n \in \N}$. We use the estimates discussed above to verify that for each $n \in \N$, we can produce a solution $W^n$ to the equation \eqref{hjb} (with $H$ replaced by $H^n$ and $G$ replaced by $G^n$) which is $H^{-k}$-Lipschitz and which is \textit{almost} a solution to the original equation \eqref{hjb} in the viscosity sense. This allows us to conclude that for large $n$, 
\begin{align*}
    U^1 \leq W^n \leq U^2,
\end{align*}
and in particular $U^1 \leq U^2$ as required.

%\textcolor{blue}{Maybe add the first two steps of the induction argument in case $d=1$ and $H = \frac{1}{2}|p|^2$.}

\subsection{Outlook}

There are a number of natural questions about the broader applicability of our approach, which we briefly address here. First, we expect that it should be possible to extend our results to non-compact state space (i.e. replace $\T^d$ with $\R^d$). The key modifications would be to work with weighted Sobolev spaces, and to use a variational principle as in e.g. \cite{bayraktaretalvariational} in the doubling of variables argument. Also, there should be no problem with allowing non-constant (but uncontrolled ) idiosyncratic noise, which amounts to replacing the idiosyncratic noise operator \eqref{idionoiseop} with an operator of the form 
\begin{align*}
    \int_{\T^d} \tr\big(a(t,x,m) D_x D_m U(t,m,x) \big) m(dx),
\end{align*}
provided that $a$ is smooth enough and non-degenerate. 

On the other hand, we rely heavily on the non-degeneracy of the idiosyncratic noise, since without this the key estimate \eqref{uniformintro:factor} is in general false, even if all the data is smooth. Thus, extending our results to allow for degenerate noise seems to require new ideas. For the same reason, it is not clear how to handle fully non-linear equations using our methods. Also, because we use the fact that the common noise operator can be understood as a finite-dimensional derivative, it is important that the common noise have constant coefficients, which is certainly a major restriction (which also appears in \cite{bayraktar2023}, the only previous paper to treat an equation with both common and idiosyncratic noise). It is the authors' view that treating an equation like \eqref{hjb} with variable common noise (and non-zero idiosyncratic noise) is an important open question, since it seems to require a genuinely infinite-dimensional version of Ishii's lemma.

\section{Preliminaries and main results}\label{sec:prelim}

\subsection{Basic notation}
We define $\T^d = \R^d/ \mathbb{Z}^d$ the flat $d$-dimensional torus and by $\mathcal{P}(\T^d)$ the set of Borel probability measures over it. We endow the latter with the first Wasserstein distance $\bd_1$, defined, for all $m_1, m_2 \in \mathcal{P}(\T^d)$ by
$$\bd_1(\mu, \nu) := \sup_{\varphi \text{ $1$-Lipschitz}} \int_{\T^d} \varphi d(m_1- m_2),$$
where the supremum is defined over the $1$-Lipschitz functions $\varphi : \T^d \rightarrow \R.$ With this metric, $\mathcal{P}(\T^d)$ becomes a compact space. 
Unless specified differently, semicontinuity or Lipschitz continuity for functions defined over $\mathcal{P}(\T^d)$ is understood with respect to this distance.
%For a multi-indice $\bm j = (j_1, \dots, j_d)$ with each $j_i \in \mathbb{N}$ (with $ \mathbb{N} := \{ 0,1,2, \dots \} $) and a function $f : \T^d \rightarrow \R$ we define 
%$$D^{\bm j} f := D_{x_1}^{j_1} \dots D_{x_d}^{j_d} f$$
%allowing for the case $D^{\bm j} f =f$ if $\bm j = (0, \dots,0)$. We also write $| \bm j| := j_1 + \dots j_d$ for the order of the multi-index $\bm j$. \red{This is different from notation used in Section 3}

%\red{I think the notation of this section for the multiindices is only used to define $C^k$, $H^{k}$ (in Section 5) and the related norms. We should define them in another way and keep Section 3 untouched.}

For $k \in \mathbb{N}$ we denote by $C^k = C^k(\T^d)$ the space of functions over $\T^d$ with continuous derivatives up to order $k$ and we endow it with the norm
$$ \norm{f}_{C^k} := \sum_{j \in \N_0^d, \; |j| \le k} \norm{D^{ j }f}_{L^\infty},$$
where $\N_0 = \{0,1,2,\ldots\}$, $\norm{\cdot}_{L^{\infty}}$ is the usual supremum norm, and, for $j = (j_1,j_2,\ldots,j_d) \in \N_0^d$, we set $|j| = j_1 + j_2 + \cdots + j_d$ and
\[
    D^{ j}  := \frac{\partial^{j_1}}{\partial x_1^{j_1}} \dots \frac{\partial^{j_d}}{\partial x_d^{j_d}} .
\]
We denote by $C^{-k}$ its dual space, ie the space of bounded linear functionals over $C^k$ and we denote by $\norm{\cdot}_{C^{-k}}$ the dual norm.

Following \cite{CarmonaDelarue_book_I} Chapter 5, we say that $\Phi : \mathcal{P}(\T^d) \rightarrow \R$ belongs to $C^1(\mathcal{P}(\T^d))$ if there is a jointly continuous map $\frac{\delta \Phi}{\delta m} : \mathcal{P}(\T^d) \times \T^d \rightarrow \R$ such that, for every $m_1 , m_2 \in \mathcal{P}(\T^d),$
\begin{equation}
    \Phi(m_2) - \Phi(m_1) = \int_0^1 \int_{\T^d} \frac{\delta \Phi}{\delta m} \bigl( (1-h) m_1 + hm_2,x) d(m_2-m_1)(x)dh.
\end{equation}
The linear derivative $\frac{\delta \Phi}{\delta m}$ is defined up to an additive constant and we adopt the normalizing convention (similar to \cite{DDJ_23})
\begin{equation}
\int_{\R^d} \frac{\delta \Phi}{\delta m}(m,x)dx = 0, \quad \forall m \in \mathcal{P}(\T^d).
\label{normalizationconvention}
\end{equation}
When $\frac{\delta \Phi}{\delta m}(m,.) \in C^1(\T^d)$ we define the intrinsic derivative of $\Phi$ at $m$ by
$$ D_m \Phi(m,x) = D_x \frac{\delta \Phi}{\delta m}(m,x).$$
Unlike the linear derivative, the intrinsic derivative is uniquely defined if it exists. 

For $\by \in (\T^d)^N$, $x \in \T^d$, and $i = 1,2,\ldots, N$, we frequently use the notation
\begin{align*}
        \by^{-i} = (y^1,...,y^{i-1},y^{i+1},...,y^N) \in (\T^d)^{N-1}, \quad (\by^{-i},x) = (y^1,...,y^{i-1},x,y^{i+1},...,y^N) \in (\T^d)^N.
    \end{align*}

\subsection{The notion of viscosity solution}

We now state precisely the definition of viscosity (sub/super)-solutions of \eqref{hjbz}. The first step is to make precise the space of test functions we will use.

\begin{defn}
    We denote by $\cpart = \cpart\big([0,T] \times \T^d \times \cP(\T^d) \big)$ the set of continuous functions $\Phi(t,z,m) : [0,T] \times \T^d \times \cP(\T^d) \to \R$
    such that the derivatives
    \begin{align*}
        (\partial_t \Phi, D_z \Phi, D_{zz}^2 \Phi)(t,z,m) : [0,T] \times \T^d \times \cP(\T^d) \to \R \times \R^d \times \R^{d \times d}
    \end{align*}
    as well as  
    \begin{align*}
        (D_m \Phi, D_x D_m \Phi)(t,z,m,x) : [0,T] \times \T^d \times \cP(\T^d) \times \T^d \to \R^d \times \R^{d \times d}
    \end{align*}
    exist and are continuous. We denote by $\cC^{1,2,2} = \cC^{1,2,2}\big([0,T] \times \T^d \times \cP(\T^d) \big)$ the set of $\Phi \in \cpart$ such that 
    $$D_{mm} \Phi(t,z,m,x,y) :  [0,T] \times \T^d \times \cP(\T^d) \times \T^d \times \T^d \to \R^{d \times d}$$ exists and is continuous.
\end{defn}

\begin{defn} \label{def.viscosityz}
    An upper semi-continuous function 
    $$V  : [0,T] \times \T^d \times \cP(\T^d) \to \R \cup \{ - \infty \}$$
    is called a viscosity subsolution of \eqref{hjbz} if $V(T,z,m) \leq \ov{G}(z,m)$, and for any 
    $\Phi \in \cpart$
    such that $V - \Phi$ attains a local maximum at $(t_0,z_0,m_0) \in (0,T) \times \T^d \times \cP(\T^d)$, we have 
    \begin{align} \label{subsoltest}
        \ds &- \partial_t \Phi(t_0,z_0,m_0) - \int_{\T^d} \tr(D_x D_m \Phi \big)(t_0,z_0,m_0,x)  m_0(dx) - a_0 \Delta_z \Phi (t_0,z_0,m_0) \vspace{.1cm} \nonumber \\
    \ds &\hspace{2cm} + \int_{\T^d} \ov{H}\big(x, D_m \Phi(t_0,z_0,m_0,x), m_0,z_0 \big) m_0(dx) \leq 0.
    \end{align}
    Similarly, a lower semi-continuous function 
    $$V  : [0,T] \times \T^d \times \cP(\T^d) \to \R \cup \{ + \infty \}$$ is called a viscosity supersolution of \eqref{hjbz} if $V(T,z,m) \geq \ov{G}(z,m)$, and for any $\Phi \in \cpart$ such that $V - \Phi$ attains a local minimum at $(t_0,z_0,m_0) \in (0,T) \times \T^d \times \cP(\T^d)$, we have 
    \begin{align} \label{supersoltest}
        \ds &- \partial_t \Phi(t_0,z_0,m_0) - \int_{\T^d} \tr(D_x D_m \Phi \big)(t_0,z_0,m_0,x)  m_0(dx) - a_0 \Delta_z \Phi (t_0,z_0,m_0) \vspace{.1cm} \nonumber  \\
    \ds &\hspace{2cm} + \int_{\T^d} \ov{H}\big(x, z_0,D_m \Phi(t_0,z_0,m_0,x), m_0 \big) m_0(dx) \geq 0.
    \end{align}
    A function $V$ is called a viscosity solution if it is both a viscosity subsolution and a viscosity supersolution.
\end{defn}

Just like in finite dimensions, in order to check that a function $V$ is a (sub/super)-solution to \eqref{hjbz}, it suffices to restrict attention to test functions with some extra regularity, and such that the maximum achieved is strict and global. We record this fact as a lemma, whose proof is delayed until the end of the section.
%\red{Suggestion: state Lemma \ref{lem.testfunctions} here, and delay its proof, as well as the statement and proof of Lemma \ref{lem.c2moll}, to the end of this section. Otherwise it takes a long time to state our main results. JJ: done, this is as good idea.}

\begin{lem} \label{lem.testfunctions}
    A function $V : [0,T] \times \T^d \times \cP(\T^d)$ is a viscosity subsolution of \eqref{hjbz} if and only if $V(T,z,m) \leq \ov{G}(z,m)$ and for each $\Phi \in C^{1,2,2}$ such that $U - \Phi$ attains a strict global maximum at $(t_0,z_0,m_0) \in (0,T) \times \T^d \times \cP(\T^d) $, \eqref{subsoltest} holds. An analogous characterization holds for viscosity subsolutions.
\end{lem}

With Definition \ref{def.viscosityz} in place, we can define vicosity solutions for the original PDE \eqref{hjb} in terms of viscosity solutions for \eqref{hjbz}.
\begin{defn} \label{def.viscosity}
    An upper semi-continuous function $V : [0,T] \times \cP(\T^d) \to \R \cup \{- \infty\}$ is called a viscosity subsolution of \eqref{hjb} if the function $\ov{V} : [0,T] \times \T^d \times \cP(\T^d) \to \R$ defined by 
    \begin{align} \label{def.ovv}
        \ov{V}(t,z,m) = V(t,m^z)
    \end{align}
    is a viscosity subsolution of \eqref{hjbz}. Likewise, a lower semi-continuous function $V : [0,T] \times \cP(\T^d) \to \R \cup \{+ \infty\}$ is called a viscosity supersolution of \eqref{hjb} if the function $\ov{V}$ defined by \eqref{def.ovv} is a viscosity supersolution of \eqref{hjbz}. 
\end{defn}

\subsection{Main results}

The following assumption will be used throughout the paper.

\begin{assumption} \label{assump.d1}
The function $G$ is Lipschitz with respect to the metric $\bd_1$ with Lipschitz constant $C_G > 0$, while $H$ satisfies the regularity condition
\be \label{assump.H} 
|H(x,p,m) - H(x',p',m')| \leq C_H \big(1 + |p| + |p'| \big) \big( |x-x'| + |p-p'| + \bd_1(m,m') \big) 
\ee
for some $C_H > 0$ each $x,x' \in \T^d$, $p,p' \in \R^d$, $m,m' \in \cP(\T^d)$.
\end{assumption}

We are now in a position to state our main results, which show that under Assumption \ref{assump.H}, there is a unique viscosity solution of \eqref{hjb} which arises as the uniform limit of the viscosity solutions of \eqref{hjbn}, and moreover the comparison principle holds for \eqref{hjb}.

\begin{thm}[Existence] \label{thm.exist}
    Suppose that Assumption \ref{assump.d1} holds. Then there exists a viscosity solution $U$ of \eqref{hjb}, which satisfies the estimate 
    \begin{align} \label{d1regintro}
        |U(t,m) - U(t',m')| \leq C \Big(|t - t'|^{1/2} + \bd_1(m,m') \Big)
    \end{align}
    for some constant $C = C(C_G, C_H)$ and all $t,t' \in [0,T]$, $m,m' \in \cP(\T^d)$.
\end{thm}

\begin{thm}[Uniqueness/comparison] \label{thm.comparison} Suppose that Assumption \ref{assump.d1} holds. Suppose that $U^1$ is a viscosity subsolution of \eqref{hjb} and $U^2$ is a viscosity supersolution of \eqref{hjb}. Then 
\begin{align*}
    U^1 \leq U^2 \text{ in } [0,T] \times \cP(\T^d).
\end{align*}
In particular, there is at most one viscosity solution of \eqref{hjb}.    
\end{thm}

\begin{thm}[Convergence] \label{thm.convergence}
    Suppose that Assumption \ref{assump.d1} holds. Then we have 
    \begin{align*}
        \sup_{(t,\bx) \in [0,T] \times (\T^d)^N} \big|V^N(t,\bx) - U(t,m_{\bx}^N) \big| \xrightarrow{N \to \infty} 0, 
    \end{align*}
    where $V^N$ denotes the unique viscosity solution of \eqref{hjbn} and $U$ denotes the unique viscosity solution of \eqref{hjb}.
\end{thm}

Our final main result states that the equation \eqref{hjb} preserves Lipschitz regularity with respect to various metrics. For this, we will need to make use of the conditions
\be \label{glipk} \tag{$\text{GLip}_k$}
\text{$G$ is $C_{G,k}$-Lipschitz with respect to $C^{-k}$.}
\ee 
\be \label{hlipk} \tag{$\text{HLip}_k$}
\begin{cases} 
\text{For each $R > 0$, there is a constant $C_{H,k,R} > 0$ such that} \vspace{.1cm}\\
      \qquad   |H(x,p,m) - H(x,p,m')| \leq C_{H,k,R} \|m - m'\|_{C^{-k}}, \vspace{.1cm} \\
      \text{for each $x \in \T^d$, $p \in B_R \subset \R^d$, $m,m' \in \cP(\T^d)$,  and } \vspace{.1cm} \\
      \qquad \sup_{m \in \cP(\T^d)} \| H(\cdot, \cdot, m) \|_{C^k(\T^d \times B_R)} \leq C_{H,k,R},
\end{cases}
\ee

\begin{thm}[Regularity] \label{thm.regularity}
    Suppose that Assumption \ref{assump.d1} holds. Suppose further that for some $k \in \N$, \eqref{glipk} and \eqref{hlipk} hold. Then the unique viscosity solution $U$ of \eqref{hjb} satisfies
    \begin{align} \label{regintro}
        |U(t,m) - U(t,m')| \leq C \|m- m'\|_{C^{-k}}
    \end{align}
    for some constant $C$ and each $t \in [0,T]$, $m,m' \in \cP(\T^d)$. Moreover, if $k \geq 2$, then $U$ is also globally Lipschitz in time.
\end{thm}

\begin{rmk} \label{rmk.nocommon}
    It is clear from the arguments below that when $a_0 = 0$, there is no need to define viscosity (sub/super)solutions in terms of the change of variables \eqref{def.ovu}. More precisely, suppose that $a_0 = 0$ and define viscosity (sub/super)solutions as follows: 

    An upper semi-continuous function $V : [0,T] \times \cP(\T^d) \to \R \cup \{- \infty\}$ is said to be a subsolution of \eqref{hjb} if for any $\Phi \in C^{1,2}_{\text{p}}([0,T] \times \cP(\T^d))$ such that $U - \Phi$ attains a local maximum at $(t_0,m_0) \in (0,T) \times \cP(\T^d)$, we have
    \begin{align*}
        - \partial_t \Phi(t_0,m_0) - \int_{\T^d} \tr(D_x D_m \Phi)(t_0,m_0,x) m_0(dx) + \int_{\T^d} H(x,D_m \Phi(t_0,m_0,x),m_0) m_0(dx) \leq 0,
    \end{align*}
    where the space $C^{1,2}_{\text{p}}([0,T] \times \cP(\T^d))$ is defined similarly to $C^{1,2,2}_{\text{p}}$, i.e. it is the space of test functions admitting continuous derivatives $\partial_t \Phi$, $D_m \Phi$, $D_x D_m \Phi$. Define supersolutions and solutions similarly. Then, the conclusions of Theorems \ref{thm.comparison}, \ref{thm.convergence}, \ref{thm.exist}, \ref{thm.regularity} still hold with this more standard definition of viscosity (sub/super)solutions.
\end{rmk}

\subsection{Proof of Lemma \ref{lem.testfunctions}}

In order to show that the definition of viscosity solutions is insensitive to the class of test functions we employ, we will need to make use of the following mollification lemma, which follows along the lines of the construction given in \cite{cossomartini} (though the argument here is simpler since we work on the torus).

\begin{lem} \label{lem.c2moll}
    Let $\Phi \in \cpart$. There exists a sequence $(\Phi^N)_{N \in \N}$ with each $\Phi^N \in \mathcal{C}^{1,2,2}$ such that 
    \begin{enumerate}
        \item $\Phi^N \to \Phi$, $D_z \Phi^N \to D_z \Phi$, $D_{zz}^2 \Phi^N \to D_z^2 \Phi$, uniformly on $[0,T] \times \T^d \times \cP(\T^d)$, \\
        \item $D_m \Phi^N \to D_m \Phi$, $D_x D_m \Phi^N \to D_x D_m \Phi$, uniformly on  $[0,T] \times \T^d \times \cP(\T^d) \times \T^d$.
    \end{enumerate}
\end{lem}

\begin{proof}
    Let $\kappa$ be a concave modulus of continuity such that 
    \begin{align}\label{kappamod}
        |D_m \Phi(t,z,m,x) - D_m \Phi(t,z,m',x')| &\leq \kappa\big( \bd_1(m,m') + |x - x'| \big), \nonumber \\
        |D_x D_m \Phi(t,z,m,x) - D_x D_m \Phi(t,z,m',x')| &\leq \kappa\big( \bd_1(m,m') + |x - x'| \big)
    \end{align}
    for each $t \in [0,T]$, $z,x,x' \in \T^d$, $m,m' \in \cP(\T^d)$.
    In order to construct $\Phi^N$, we first define for each $N \in \N$, $\Psi^N : [0,T] \times \T^d \times (\T^d)^N \to \R$ by 
    \begin{align*}
        \Psi^N(t,z,\bx) = \Phi(t,z,m_{\bx}^N). 
    \end{align*}
    Then let $(\rho_{\epsilon})_{\epsilon > 0}$ be a standard mollifier on $\R^d$, i.e. $\rho_{\epsilon}(x) = \epsilon^{-d} \rho(x/\epsilon)$ for some smooth, non-negative $\rho : \R^d \to \R$ such that $\int_{\R^d} \rho(x) dx = 1$ and $\text{supp}(\rho) \subset B_{1} \subset \R^d$, and set
    \[
       \boldsymbol{\rho}^N_\epsilon(\bx) = \prod_{k=1}^N \rho_\epsilon(x^k). 
   \]
    Define for each $\epsilon > 0$
    \begin{align*}
        \Psi^{N,\epsilon}(t,z,\bx) = \int_{(\R^d)^N} \Psi^N(t,z,\by) \bm{\rho}_\epsilon^N(\bx - \by) d\by. 
    \end{align*}
    Next, define
    \begin{align*}
        \Phi^N(t,z,m) = \int_{(\T^d)^N} \Psi^{N,\epsilon_N}(t,z,\bx) m^{\otimes N}(d\bx),
    \end{align*}
    where $\epsilon_N \downarrow 0$ are chosen so that 
    \begin{align*}
        \epsilon_N^{-1} \kappa(1/N) \xrightarrow{N \to \infty} 0.
    \end{align*}
    It is clear that $\Phi^N$ is smooth in $m$, since $\Psi^{N,\epsilon}$ is smooth in $\bx$, and so one checks easily that $\Phi^N \in C^{1,2,2}$. We now verify the convergence statements in (1) and (2). In fact, we focus on verifying 
    \begin{align*}
        D_x D_m \Phi^N \to D_x D_m \Phi,
    \end{align*}
    since this is the most challenging convergence statement to verify, and the convergence of the other derivatives can be established along similar lines.

    First, we compute 
    \begin{align} \label{dxdmcomp}
        D_m \Phi^N(t,z,m,x) &= \sum_{i = 1}^N \int_{(\T^d)^{N-1}} D_{x^i} \Psi^{N,\epsilon_N}\big(t,z, (\by^{-i}, x) \big) dm^{\otimes (N-1)} (d \by^{-i}), \nonumber \\
        D_x D_m \Phi^N(t,z,m,x) &= \sum_{i = 1}^N \int_{(\T^d)^{N-1}} D_{x^ix^i} \Psi^{N,\epsilon_N}\big(t,z, (\by^{-i}, x) \big) dm^{\otimes (N-1)} (d \by^{-i}).
    \end{align}
  %  where we use the notation 
 %   \begin{align*}
  %      \by^{-i} = (y^1,...,y^{i-1},y^{i+1},...,y^N) \in (\T^d)^{N-1}, \quad (\by^{-i},x) = (y^1,...,y^{i-1},x,y^{i+1},...,y^N) \in (\T^d)^N.
 %   \end{align*}
    Next, we notice that
    \begin{align*}
        D_{x^ix^i} \Psi^{N,\epsilon}(t,z,\bx) &= \frac{1}{N} \int_{(\T^d)^N} D_m \Phi(t,z,m_{\by}^N,y^i)  \otimes  D_{x^i} \bm{\rho}^N_\epsilon(\bx - \by) d\by\\
        &= \frac{1}{N} \int_{(\T^d)^N} D_m \Phi(t,z,m_{\by^{-i}}^{(N-1)},y^i) \otimes D_{x^i} \bm{\rho}^N_\epsilon(\bx - \by) d\by
        + r^{N,\epsilon}(t,z,\bx),
    \end{align*}
    where, in view of \eqref{kappamod},
    \[
        |r^{N,\epsilon}(t,z,\bx)| \le\frac{1}{N} \epsilon^{-1} \kappa(1/N).
    \]
    Therefore,
    \begin{align} \label{technical}
        |D_{x^ix^i} &\Psi^{N,\epsilon}(t,z,\bx) - \frac{1}{N} D_x D_m \Phi(t,z,m_{\bx}^N,x^i)| \leq 
       \frac{1}{N} \epsilon^{-1} \kappa(1/N) \nonumber  \\
       &+ \frac{1}{N} \bigg|  \int_{(\T^d)^N} D_m \Phi(m_{\by^{-i}}^{(N-1)},y^i) D_{x^i} \bm{\rho}^N_\epsilon(\bx - \by) d\by 
       - D_xD_m \Phi(t,z,m_{\bx}^N,x^i) \bigg| \nonumber  \\ 
       &= \frac{1}{N} \epsilon^{-1} \kappa(1/N) + \frac{1}{N} \bigg|  \int_{(\T^d)^N} D_x D_m \Phi(t,z,m_{\by^{-i}}^{(N-1)},y^i) \bm{\rho}_{\epsilon}^N(\bx - \by)  d\by 
       - D_x D_m \Phi(t,z,m_{\bx}^N,x^i) \bigg| \nonumber \\
       &\leq \frac{C}{N} \bigg(\epsilon^{-1} \kappa(1/N) +  \kappa\big(C/N + C\epsilon \big)\bigg), 
    \end{align}
    where $C$ is independent of $\epsilon$ and $N$ and we have used the fact that $\text{supp } \rho_{\epsilon} \subset B_{\epsilon}$, $\| D_{x^i} \rho_{\epsilon} \|_{L^1} \leq C/\epsilon$, and the fact that $\bd_1(m_{\bx}^N, m_{\bx^{-i}}^{N-1}) \leq C/N$. To complete the proof, we combine \eqref{dxdmcomp} and \eqref{technical} to get 
    \begin{align*}
        |D_x D_m& \Phi(t,z,m,x) - D_x D_m \Phi^N(t,z,m,x)| \leq C \Big(\epsilon_N^{-N} \kappa(1/N) +  \kappa\big(C/N + C\epsilon_N \big)\Big) \\
        &+ \frac{1}{N} \sum_{i = 1}^N \bigg| \int_{(\T^d)^{N-1}} D_x D_m \Phi(t,z,m^{N}_{(\by^{-i},x)},x)  dm^{\otimes (N-1)} (\by^{-i}) - D_x D_m \Phi(t,z,m,x) \bigg| \\
        &\leq C \Big(\epsilon_N^{-1} \kappa(1/N) +  \kappa\big(C/N + C\epsilon_N \big) + \kappa(r_{N-1,d}) \Big), 
    \end{align*}
    where \begin{align*}
        r_{N,d} = \begin{cases}
            N^{-1/d} & d \geq 3 \\
            N^{-1/2} \log(N) & d = 2 \\
            N^{-1/2} & d = 1,
        \end{cases}
    \end{align*} and we have used the concavity of $\kappa$ and the main results of \cite{fournier2015rate} to obtain the final estimate. In light of the fact that $\epsilon_N^{-N} \kappa(1/N) \to 0$ and $\epsilon_N \to 0$ by construction, this completes the proof that $D_x D_m \Phi^N \to D_x D_m \Phi$ uniformly.
\end{proof}

We now complete the proof of Lemma \ref{lem.testfunctions}.

\begin{proof}[Proof of Lemma \ref{lem.testfunctions}]
    The fact that we can restrict our attention to strict global maxima is justified by standard arguments, e.g. replacing a test function $\Phi$ with one of the form 
    \begin{align*}
        (t,z,m) \mapsto \Phi(t,z,m) + \frac{1}{\epsilon} \Big(\| m - m_0\|_{-k}^2 + |z - z_0|^4 + |t - t_0|^2 \Big)
    \end{align*}
    for some large $k \in \N$ and small $\epsilon > 0$, where $\| \cdot \|_{-k}$ denotes the negative Sobolev norm introduced in subsection \ref{sec.comparison}, and we use the fact that, for $k$ sufficiently large, $\left\{ m \mapsto \norm{m - m_0}^{2}_{-k} \right\} \in \mathcal C^1(\mathcal P(\T^d))$ and 
    \[
        D_m \left[\norm{m - m_0}^{2}_{-k}\right](x)
        = D_x D_m \left[\norm{m - m_0}^{2}_{-k}\right](x) \equiv 0 \quad \text{at }m = m_0.
    \]
    The fact that we can restrict to $\Phi \in C^{1,2,2}$ is then easily justified using Lemma \ref{lem.c2moll}.
\end{proof}

\section{Uniform estimates for \eqref{hjbn}} \label{sec.uniformest}

The purpose of this section is to establish precise estimates on the derivatives of the solution $V^N$ of \eqref{hjbn} that will be used to prove $C^{-k}$-Lipschitz-regularity for the limiting equation. For the rest of this section, we use the notation
\begin{equation}\label{R}
    R^* := e^{2C_H T}C_G, \quad C_{H,k} \coloneqq C_{H,k,R^*},
\end{equation}
where $C_H$, $C_G$, and $C_{H,k,R}$ are as in Assumption \ref{assump.d1} and \eqref{hlipk}.

In order to state the result, we first prove the following.

\begin{lem}\label{lem:VnLip}
    Suppose that Assumption \ref{assump.d1} holds. Then, for all $N$ and $t \in [0,T]$, we have
    \[
        \max_{k =1,\ldots,N} \norm{D_{x^k}V^N(t,\cdot)}_\infty \le \frac{1}{N}R^*.
    \]
    In addition, for all $0 \leq s < t \leq T$, $x \in \T^d$, we have
    \[
    |V^N(t,x) - V^N(s,x)| \leq C \sqrt{t-s}, 
    \]
    where $C$ depends only on $R^*$ and $C_{H,0} = \|H\|_{\linf(\T^d \times B_{R^*} \times \cP(\T^d))}$.
\end{lem}

%\red{JJ: We need to pick between subscripts and superscripts, i.e. $x^n$ or $x_n$. I think we are using $x^n$ in some other places.}

Here and in the rest of this section, we introduce the elliptic, self-adjoint operator
\[
    \mathcal L^N  = \sum_{i=1}^N \Delta_{x^i} + a_0 \sum_{i,j=1}^N \tr[D_{x^i x^j}],
\]
and we note that, for any smooth scalar function $v: (\T^d)^N \to \R$,
\begin{equation}\label{quad:gain}
    \mathcal L^N v^2 \ge 2v \mathcal L^N v + 2 |D_{\bx}v|^2
\end{equation}
(equality holds if $a_0 =0$).

\begin{proof}[Proof of Lemma \ref{lem:VnLip}]
    Observe first that 
    \begin{equation}\label{DG:bound}
        \norm{D_{x^k} V^N(T,\cdot)}_\infty \le \frac{C_G}{N}.
    \end{equation}
    Set $V^{N,k} = D_{x^k}V^N$. Differentiating \eqref{hjbn} in $x^k$ yields
    \begin{align}
        -\partial_t V^{N,k} &- \mathcal L^N V^{N,k} - \sum_{i=1}^N \left[ D_p H(x^i, ND_{x^i}V^N, m^N_\bx)\cdot D_{x^i}\right] V^{N,k} \notag \\
        &= - \frac{1}{N} D_x H(x^k, ND_{x^k}V^N,m^N_\bx) 
        - \frac{1}{N^2}\sum_{i=1}^N D_m H(x^i, ND_{x^i}V^N, m^N_\bx, x^n).\label{VNkeq}
    \end{align}
%\textcolor{red}{I guess it should be $\sum_{i=1}^N D_{x^{i}} V^{N,n} D_pH(x_{i},ND_{x_{i}}N^N,m_{\bx}^N)$ as soon as $d \geq 1$.}

%\red{I left it as is because I think it is more natural, especially later in connection with the FK eq: we are differentiating $V^{N,k}$ in the direction of the vector field $(D_pH, D_pH, \ldots, D_pH)$. I put some brackets to make it clearer though. SD: fair enough}
%    \red{Integrate the equation above against the solution to \eqref{dual:FP} (but with inital condition a Dirac) to get 
%    \begin{align*}
%        \|V^{N,k}(t_0,\cdot)\|_{\infty} \leq C_0/N + C \int_{t_0}^T \bigg( \|V^{N,k}(t, \cdot)\|_{\infty} + \frac{1}{N} \sum_{j} \|V^{N,j}(t,\cdot) \|_{\infty} \bigg) dt, 
%    \end{align*} 
 %   so that the function 
  %  \begin{align*}
  %      t \mapsto w(t) = \max_{k = 1,...,N} \|V^{N,k}(t,\cdot)\|_{\infty}
  %  \end{align*}
  %  satisfies 
  %  \begin{align*}
  %      w(t) \leq C_0/N + C \int_{t}^T w(s) ds, 
  %  \end{align*}
  %  and conclude by Gronwall that $D_{x^i} V^{N,k} \leq C/N$.}
    
    In view of \eqref{assump.H}, 
    \begin{align*}
        &\abs{ \frac{1}{N} D_x H(x^k, ND_{x^k}V^N,m^N_\bx) +
        \frac{1}{N^2}\sum_{i=1}^N D_m H(x^i, ND_{x^i}V^N, m^N_\bx, x^k)}\\
        &\le C_H\left( |D_{x^k}V^N| + \frac{1}{N}\sum_{i=1}^N |D_{x^i}V^N| \right)
        \le 2C_H \max_{k = 1,\ldots,N} |V^{N,k}|.
    \end{align*}
    Let $W = \max_{k=1,\ldots,N} |V^{N,k}|^2$. Then, by \eqref{quad:gain} and the convexity of the $\max$ function, it follows that
    \[
        -\partial_t W - \mathcal L^N W + \sum_{i=1}^N D_p H(x^i, ND_{x^i}V^N, m^N_\bx)\cdot D_{x_i}W \le 4C_H W.
    \]
    The estimate on $\max_{k} \|D_{x^k} V^N(t,\cdot)\|_{\infty}$ now follows from the maximum principle and Gr\"onwall's inequality.

    To complete the proof, let $(P^N_h)_{h > 0}$ denote the solution operator associated to the operator $\mathcal{L}^N$, and observe that, for some constant $C > 0$ depending only on $a_0$ and $d$ (but not on $N$), for any smooth function $V: (\T^d)^N \to \R$ and $h > 0$,
    \[
        \norm{P^N_h V - V}_\infty \le CNh^{1/2} \sup_{k=1,\ldots,N} \norm{D_{x^k}V}_\infty.
    \]
    From the Lipschitz estimate proved above, we have
    \[
        -\partial_t V^N - \mathcal{L}^N V^N = f^N
    \]
    for some $f^N: (\T^d)^N \to \R$ satisfying $\norm{f^N}_\infty \le C_{H,0}$. By Duhamel's formula, for any $t_0 \in [0,T)$ and $h \in (0,T-t_0)$,
    \begin{align*}
        V^N(t_0,\cdot) = P^N_{h} V^N(t+h, \cdot) + \int_{t_0}^{t_0+h} P^N_{t - t_0}f^N dt, 
    \end{align*}
    and therefore, in view of the Lipschitz estimate,
    \begin{align*}
        |V^N(t_0,\cdot) - V^N(t_0 + h, \cdot)| \leq C_{H,0} h + |P^N_h  V^N(t+h, \cdot) - V^N(t+h,\cdot)| \leq C \sqrt{h}, 
    \end{align*}
    with $C$ as described in the statement of the lemma.
\end{proof}

We can now state the main result of this section.

\begin{thm}\label{thm:VNestimates}
    In addition to Assumption \ref{assump.d1}, assume that, for some $k \ge 1$, $G$ and $H$ are $k$-times continuously differentiable in all variables, and satisfy \eqref{glipk} and \eqref{hlipk}. Then there exist constants $C_0 = C_{0,k} > 0$ and $C_1 = C_{1,k} > 0$ independent of $N$ such that $C_{0}$ depends on the data only through the quantities $C_{H,k}$ and $C_{G,k}$, such that, for any $N \in \N$ and $n \in \{1,2,\ldots,N\}$,
    \[
        \norm{D_{x^n}^k V^N}_\infty \le \frac{C_0}{N} + \frac{C_1}{N^{3/2}}.
    \]
\end{thm}

\begin{rmk}
    We note that the constant $C_1$ is allowed to depend on bounds for derivatives of $H$ and $G$ in the measure variable that are higher than first-order. However, it is crucial for our main propagation-of-regularity result (Theorem \ref{thm.regularity}) that $C_0$ only depend on $C_{H,k}$ and $C_{G,k}$, which bound first-order derivatives of $H$ and $G$ in $m$. Then, as $N \to \infty$, the influence of the ``bad'' constant $C_1$ vanishes; indeed, see the proof of Proposition \ref{prop.existsmooth} below.
\end{rmk}

%\textcolor{red}{It would be nice to add a comment at this point. In particular, estimating mixed derivatives seems in contradiction with the explanation given in the introduction. Maybe it could also be emphasized that, "while $C_1$ can depend on ... and ..., it is very important that $C_0$ depends on the data only through ..." and then refer to Proposition \ref{prop.existsmooth} }

We define the smooth vector field $\bm{b}^N: [0,T] \times (\T^d)^N \to (\T^d)^N$ by
\[
    \bm b^N(t,\bx) =  (- D_p H(x^i, ND_{x^i}V^N(t,\bx),m^N_\bx) )_{i=1}^N,
\]
which is uniformly bounded in view of Lemma \ref{lem:VnLip}.

Given $t_0 \in [0,T]$ and $f \in C_c^\infty((\T^d)^N)$, denote by $\phi = \phi^{N,t_0,f}: [t_0,T] \times (\T^d)^N \to \R$ the smooth solution of the Fokker-Planck equation
\begin{equation}\label{dual:FP}
    \begin{cases}
        \ds \partial_t \phi - \mathcal L^N \phi + \div_{\bx} ( \bm b^N(t,\bx) \phi) = 0,& t \in [t_0,T]\\
        \phi(t_0,\cdot) = f.
    \end{cases}
\end{equation}
Then the following duality result is established through a simple computation.

\begin{lem}\label{lem:duality}
    Let $r \in L^\infty_+([0,T] \times (\T^d)^N)$ and let $W \in C([0,T] \times (\T^d)^N)$ be nonnegative and satisfy
    \[
        -\partial_t W - \mathcal L^N W - \bm b^N(t,\bx) \cdot D_\bx W \le r \quad \text{in } [0,T] \times (\T^d)^N
    \]
    in the sense of distributions. Then, for every $t_0 \in [0,T]$ and $f \in L^1_+((\T^d)^N)$,
    \[
        \int_{(\T^d)^N} W(t_0,\bx)f(\bx)d\bx \le \int_{(\T^d)^N} W(T,\bx)\phi^{N,t_0,f}(T,\bx)d\bx + \int_{t_0}^T \int_{(\T^d)^N} r(t,\bx) \phi^{N,t_0,f}(t,\bx)d\bx dt.
    \]
\end{lem}

For $k =1,2,\ldots$, we denote multiindices by $\bm n = (n_1,n_2,\ldots,n_k)$ and $|\bm n| = k$. For such $\bm n$, define
\[
    V^{N,\bm n} := \left(\prod_{i=1}^k D_{x^{n_i}}\right) V^{N}.
\]
%\textcolor{red}{choice of notation for the derivatives, cf \eqref{Gderiv}}

The proof of Theorem \ref{thm:VNestimates} will follow immediately from the following stronger statement. Let us emphasize that, in the statements below, $|\cdot|$ denotes the standard Euclidean norm on the relevant finite-dimensional space. In particular, for a function $W: (\T^d)^N \to \R$,
\[
    |DW|^2 = \sum_{i=1}^N |D_{x^i}W|^2.
\]

\begin{lem}\label{lem:induction}
    Let $k = 1,2,\ldots$ and let $\bm n$ be a multiindex of size $k$. Then there exist constants $C_0$ and $C_1$ as in the statement of Theorem \ref{thm:VNestimates} such that
    \[
        \norm{V^{N,\bm n}}_\infty \le \frac{C_0}{N} + \frac{C_1}{N^{3/2}}
    \]
    and, for any $f \in L^1_+((\T^d)^N)$ and $t_0 \in [0,T]$,
    \[
        \int_{t_0}^T \int_{(\T^d)^N} |D V^{N,\bm n}(t,\bx)|^2 \phi^{N,t_0,f}(t,\bx) d\bx dt \le  \left(\frac{C_0}{N} + \frac{C_1}{N^{3/2}}\right)^2 \norm{f}_{L^1}.
    \]
\end{lem}

\begin{rmk}
    When $n \in \{1,2,\ldots,N\}$ is fixed and $n_1 = \cdots = n_k = n$, we obtain an estimate on $D_{x^n}^k V^N$. We must prove the more general statement above involving mixed derivatives in order to close the induction argument.
\end{rmk}

\begin{proof}[Proof of Lemma \ref{lem:induction}]
    The proof proceeds by strong induction on $k$. For the base case, we already have the uniform bound on $V^{N,n}$ for any $n = 1,2,\ldots,N$ by Lemma \ref{lem:VnLip}. Using \eqref{quad:gain}, we arrive at, for some $C_0$ as in the statement of the lemma,
    \[
    -\partial_t |V^{N,n}|^2 - \mathcal L^N |V^{N,n}|^2 + 2  |D_{\bx}V^{N,n}|^2 - \bm b^N(t,\bx) \cdot D_\bx |V^{N,n}|^2 \le \frac{C_0^2}{N^2},
    \]
    where we have now used Lemma \ref{lem:VnLip} to bound the right-hand side. Choosing $f \in L^1_+((\T^d)^N)$ and $t_0 \in [0,T]$ and using \eqref{DG:bound} and Lemma \ref{lem:duality}, we finish the proof of the base case by estimating
    \begin{align*}
        &\int_{(\T^d)^N} |V^{N,n}(t_0,\bx)|^2 f(\bx)d\bx
        + \int_{t_0}^T \int_{(\T^d)^N} |D_{\bx}V^{N,n}(t,\bx)|^2 \phi^{N,t_0,f}(t,\bx)d\bx dt\\
        &\le \int_{(\T^d)^N} |V^{N,n}(T,\bx)|^2\phi^{t_0,f}(T,\bx)d\bx + \frac{C_0^2(T-t_0)}{N^2}\norm{f}_{L^1}\\
        &\le \left(\norm{V^{N,n}(T,\cdot)}_\infty^2 + \frac{C_0^2(T-t_0)}{N^2} \right)\norm{f}_{L^1}
        \le \frac{C_0'}{N^2} \norm{f}_{L^1},
    \end{align*}
    where we used the fact that $\norm{\phi(t,\cdot)}_{L^1} = \norm{f}_{L^1}$ for all $t \in [t_0,T]$. 
    
    Assume now that the estimates hold whenever $|\bm n| = 1,2,\ldots, k$. Let then $\bm n = (n_1,n_2,\ldots, n_{k+1})$ be a multiindex of length $k+1$. We now differentiate \eqref{hjbn} in $x^{n_j}$ for $j = 1,2,\ldots, k+1$, and keep track of the various terms that arise from differentiating the nonlinearity through the use of the Fa\`a di Bruno formula. 

    First, when one derivative falls in the $p$-variable in $H$, and all others on the term $D_{\bx}V^{N,n_i}$ that comes out, we arrive at exactly
    \[
        -\bm b^{N}(t,\bx) \cdot D_\bx V^{N,\bm n}.
    \]

    All other terms will involve lower order derivatives of $V^N$. Let us first list those terms that are not multiplied by any derivatives of $V^N$, that is, the terms involving only derivatives in the non-$p$ variables of $H$. When all derivatives fall in the $x$-variable, we obtain the term
    \[
        \I_1 = \frac{1}{N}(D^k_x H)(x^{n_1}, ND_{x^{n_1}} V^N, m^N_\bx) \delta_{n_1=n_2=\cdots=n_k}.
    \]
    If one derivative hits the $m$ variable and all others fall on the auxiliary variables that arise from that, we obtain the term
    \[
        \I_2 = \frac{1}{N^2}\sum_{i=1}^N \left(D_y^{k-1} \frac{\partial H}{\partial m}\right)(x^i, ND_{x^i}V^N,m^N_\bx, x^{n_1}) \delta_{n_1 = n_2 = \cdots=n_k},
    \]
    and we conclude, for some $C_0$ as in the statement of the lemma,
    \[
        |\I_1| + |\I_2| \le \frac{C_0}{N}.
    \]
    All other terms involving derivatives not in the $p$-variable will result in factors of $1/N^2$.
    
    Let us next focus on terms consisting of derivatives of $H$ multiplied by a single derivative of $V^N$. In other words, these are the terms where exactly one $p$-derivative of $H$ is taken. One such term, already accounted for above with the vector field $\bm b^N$, is the one in which the \emph{only} derivative of $H$ is one in the $p$ variable. All other such terms will take the form
    \[
        \II_{\bm n'} = \frac{1}{N} \sum_{i=1}^N c^{i,N,\bm n'}(t,\bx) \cdot ND_{x^i}V^{N,\bm n'}(t,\bx),
    \]
    where $\bm n'$ is a multi-index of size $|\bm n'| < k+1$ and $c^{i,N,\bm n'}$ results from taking exactly one derivative in the $p$-variable of $H$ and all other derivatives in the other variables, and therefore, for some $C_1$ as in the statement of the lemma,
    \[
        |c^{i,N,\bm n'}| \le \frac{C_1}{N}.
    \]
    Invoking the inductive hypothesis and the Young and Cauchy-Schwarz inequalities, we then estimate, for any $f \in L^1_+((\T^d)^N)$, $t_0 \in [0,T]$, and $\varepsilon > 0$,
\begin{align*}
    &\abs{\int_{t_0}^T \int_{(\T^d)^N} \II_{\bm n'}(t,\bx) \phi^{N,t_0,f}(t,\bx) d\bx dt}\\
    &\le \frac{C_1}{N} \sum_{i=1}^N \int_{t_0}^T \int_{(\T^d)^N} |D_{x^i} V^{N,\bm n'}(t,\bx)| \phi^{N,t_0,f}(t,\bx)d\bx dt \\
    &\le \frac{C_1(T-t_0)^{1/2}}{N} \sum_{i=1}^N \left( \int_{t_0}^T \int_{(\T^d)^N} |D_{x^i}V^{N,\bm n'}(t,\bx)|^2 \phi^{N,t_0,f}(t,\bx)d\bx dt \right)^{1/2} \norm{f}_{L^1}^{1/2}\\
    &\le \frac{C_1(T-t_0)^{1/2}}{N} \sum_{i=1}^N \left( \frac{1}{2\varepsilon}\int_{t_0}^T \int_{(\T^d)^N} |D_{x^i}V^{N,\bm n'}(t,\bx)|^2 \phi^{N,t_0,f}(t,\bx)d\bx dt + \frac{\varepsilon}{2} \right) \norm{f}_{L^1}^{1/2}\\
    &\le \frac{C_1'}{N}\left( \frac{1}{\varepsilon N^2} \norm{f}_{L^1}^{3/2} + N \varepsilon \norm{f}_{L^1}^{1/2} \right).
\end{align*}
Optimizing in $\epsilon$ yields $\epsilon = N^{-3/2} \norm{f}_{L^1}^{1/2}$. We thus conclude that
\[
    \abs{\int_{t_0}^T \int_{(\T^d)^N} \II_{\bm n'}(t,\bx) \phi^{N,t_0,f}(t,\bx) d\bx dt}
    \le \frac{C_1}{N^{3/2}} \norm{f}_{L^1}.
\]

Finally, the remaining terms all involve derivatives of $H$ in which the $p$-variable is differentiated at least twice. Let us first focus on the terms only involving derivatives of $H$ in the $p$-variables. For some $\ell = 2,3,\ldots,k+1$ (the case $\ell=1$ is already covered by the linearized term with $\bm b^N$ above), these terms take the form, up to some multiplicative constants depending on $k$ but not $N$,
    \begin{equation}\label{pderiv:terms}
       \overline{\III} = N^{\ell-1} \sum_{i=1}^N D_p^\ell H(x^i, ND_{x^i}V^N, m^N_\bx) \bigotimes_{j=1}^{\ell} D_{x^i} V^{N,\bm n_j}
    \end{equation}
    where $|\bm n_1| + |\bm n_2| + \cdots + |\bm n_{\ell}| = k+1$. Because $\ell \ge 2$, each $\bm n_j$ can have length at most $k$, and it cannot be the case that $|\bm n_j| = k$ for more than two values of $j = 1,2,\ldots,\ell$. For any $\bm n_j$ with $|\bm n_j| < k$ we have by the inductive hypothesis
    \[
        N \norm{D_{x^i}V^{N,\bm n_j}}_\infty \le C_0 + \frac{C_1}{N^{1/2}},
    \]
    and so (changing constants as appropriate and continually using $N \ge 1$), for every term of the form \eqref{pderiv:terms}, there exist multi-indices $\bm n_1$ and $\bm n_2$ of length at most $k$ such that 
    \[
        |\overline{\III}| \le
       (C_0 N + C_1N^{1/2})  \sum_{i=1}^N |D_{x_i} V^{N,\bm n_1}| \cdot |D_{x_i}V^{N,\bm n_2}|.
    \]
    Using Young's inequality for the finite sum, we conclude that all the terms involving more than $2$ derivatives of $H$ in the $p$-variable, and no other derivatives of $H$, can be bounded by a finite linear combination (the number of which is independent of $N$) of terms of the form
    \begin{equation}\label{mainterms}
        \III_{\bm n'} = (C_0 N + C_1N^{1/2})  \sum_{i=1}^N |D_{x_i} V^{N,\bm n'}|^2
    \end{equation}
    for a multi-index $\bm n'$ of length $\le k$. The same is true if some derivatives also fall on the $x$, $m$, or the auxiliary $y$ variables, since in those instances we gain negative powers of $N$. From the inductive hypothesis,
    \[
        \int_{t_0}^T \int_{(\T^d)^N} |\III_{\bm n'}(t,\bx)| \phi^{N,t_0,f}(t,\bx)d\bx dt
        \le (C_0 N + C_1N^{1/2}) \left( \frac{C_0}{N} + \frac{C_1}{N^{3/2}} \right)^2
        \le \frac{C_0'}{N} + \frac{C_1'}{N^{3/2}}.
    \]

    Combining all terms, we find that
    \begin{equation}\label{VNk:eqform}
        -\partial_t V^{N,\bm n} - \mathcal L^N V^{N,\bm n} - \bm b^{N}(t,\bx) \cdot D_\bx V^{N,\bm n}
        = r^N,
    \end{equation}
    where, for any $t_0 \in [0,T]$ and $f \in L^1_+((\T^d)^N)$,
    \[
        \int_{t_0}^T \int_{(\T^d)^N} |r^N(t,\bx)| \phi^{N,t_0,f}(t,\bx)d\bx dt
        \le \left( \frac{C_0}{N} + \frac{C_1}{N^{3/2}}\right) \norm{f}_{L^1}.
    \]
    We also have
    %Before presenting the proof of Theorem \ref{thm:VNestimates}, we observe that, for any $k \ge 2$ and $n_1,n_2,\ldots,n_k \in \{1,2,\ldots,N\}$,
\begin{equation}\label{Gderiv}
    \left(\prod_{i=1}^{k+1} D_{x^{n_i}}\right)  G(m^N_\bx)
    = \frac{1}{N} D_x^{k+1} \frac{\partial G}{\partial m}(m^N_\bx, \cdot)\Big|_{x = x^{n_1}} \delta_{n_1=n_2=\cdots=n_{k+1}}
    + O(1/N^2),
\end{equation}
where the proportionality constant depends only on $k$ and the higher order derivatives of $G$ in the $m$-variable. By convexity, 
    \[
        -\partial_t |V^{N,\bm n}| - \mathcal L^N |V^{N,\bm n}| + \bm b^{N}(t,\bx) \cdot D_\bx |V^{N,\bm n}|
        \le |r^N|,
    \]
    and then Lemma \ref{lem:duality} yields
    \begin{align*}
        \int_{(\T^d)^N} |V^{N,\bm n}(t_0,\bx)| f(\bx)d\bx
       &\le \int_{(\T^d)^N} |V^{N,\bm n}(T,\bx)| \phi^{N,t_0,f}(T,\bx)d\bx \\
       &+ \int_{t_0}^T \int_{(\T^d)^N} |r^N(t,\bx)| \phi^{N,t_0,f}(t,\bx)d\bx dt 
        \le \left( \frac{C_0}{N} + \frac{C_1}{N^{3/2}} \right) \norm{f}_{L^1}.
    \end{align*}
    Since $f \in L^1$ was arbitrary, this concludes the proof of the supremum bound for $|V^{N,\bm n}|$ with $|\bm n| = k+1$.

    Contracting \eqref{VNk:eqform} with $2V^{N,\bm n}$, we arrive at
    \begin{align*}
        &-\partial_t |V^{N,\bm n}|^2 - \mathcal L^N |V^{N,\bm n}|^2 + 2|D_\bx V^{N,\bm n}|^2 - \bm b^N(t,\bx) \cdot D_\bx |V^{N,\bm n}|^2 \le 2|r^N||V^{N,\bm n}|\\
        &\le 2\left(\frac{C_0}{N} + \frac{C_1}{N^{3/2}}\right)|r^N|.
    \end{align*}
    Using Lemma \ref{lem:duality} once more, along with the inductive hypothesis, we conclude the proof of the inductive step.
\end{proof}

\section{Existence via finite-dimensional approximations} \label{sec.existence}

We will call a function $U = U(t,m) : [0,T] \times \pr(\T^d)$ a limit point of the sequence $V^N$ if for some subsequence $(V^{N_k})_{k \in \N}$, we have 

\begin{align*}
    \sup_{(t,\bx) \in [0,T] \times (\T^d)^{N_k}} \big|U(t,m_{\bx}^{N_k}) - V^{N_k}(t,\bx) \big| \xrightarrow{k \to \infty} 0.
\end{align*}

The first goal of this section is to prove that limit points are viscosity solutions, in the sense of Definition \eqref{def.viscosity}.

\begin{prop} \label{prop.viscosity}
Suppose that Assumption \ref{assump.d1} holds. Then any limit point $U$ of the sequence $(V^N)_{N \in \N}$ is a viscosity solution of \eqref{hjb}. 
\end{prop}

\begin{proof}
Without loss of generality, we assume that 
\begin{align} \label{vnconv}
    \sup_{(t,\bx) \in [0,T] \times (\T^d)^N} \big|U(t,m_{\bx}^N) - V^{N}(t,\bx) \big| \xrightarrow{N \to \infty} 0.
\end{align}
Define $\ov{U}(t,z,m)$ by \eqref{def.ovu}. In addition, we define for each $N \in \N$, the function $\ov{V}^N : [0,T] \times \T^d \times (\T^d)^N$ via the formula 
\begin{align*}
    \ov{V}^N(t,z,\bx) = V^N(t,\bx^z) \coloneqq V^N(t,x^1 + z,...,x^N + z). 
\end{align*}
Explicit computation shows that $\ov{V}^N$ is a classical solution of 
\begin{align} \label{hjbovn}
\begin{cases} 
\displaystyle - \partial_t \ov{V}^N - \sum_{ i = 1}^N \Delta_{x^i} \ov{V}^N - a_0 \Delta_z \ov{V}^N   \vspace{.1cm}  + \frac{1}{N} \sum_{i = 1}^N \ov{H}(x^i, z, ND_{x^i} \ov{V}^N, m^N_{\bx}) = 0, \\ \hspace{8cm} (t,z,\bx) \in (0,T) \times \T^d \times(\T^d)^N, \vspace{.1cm} \\ 
\ov{V}^N(T,z,\bx) = \ov{G}(z, m_{\bx}^N), \quad \bx \in \T^d \times (\T^d)^N. \end{cases}
\end{align}
By \eqref{vnconv}, we have
\begin{align} \label{vnbarconv}
     &\sup_{(t,z,\bx) \in [0,T] \times \T^d \times (\T^d)^N} \big|\ov{U}(t,z,m_{\bx}^N) - \ov{V}^{N}(t,z,\bx) \big| \\
     &= \sup_{(t,z,\bx) \in [0,T] \times \T^d \times (\T^d)^N} \big|U(t,m_{\bx^z}^N) - V^{N}(t,\bx^z) \big| \xrightarrow{N \to \infty} 0.
\end{align}
By Lemma \ref{lem.testfunctions}, it suffices to consider a test function $\Phi \in C^{1,2,2}$ such that $\ov{U} - \Phi$ attains a strict maximum at some $(t_0,z_0,m_0) \in (0,T) \times \T^d \times \pr(\T^d)$. Define $\Phi^N : [0,T] \times \T^d \times (\T^d)^N \to \R$ by 
\begin{align*}
    \Phi^N(t,z,\bx) = \Phi(t,z,m_{\bx}^N). 
\end{align*}
By \eqref{vnbarconv} and strict maximality, we can find $(t^N, z^N, \bx^N)$ such that 
\begin{align} \label{approxoptimizer}
    (t^N,z^N,m_{\bx^N}^N) \xrightarrow{N \to \infty} (t_0,z_0,m_0), 
\end{align}
and $\ov{V}^N - \Phi^N$ attains a maximum at $(t^N,z^N,\bx^N)$. Using the subsolution property for $\ov{V}^N$, and recalling that 
\begin{align*}
    D_{x^i x^i} \Phi^N(t,z,\bx) = \frac{1}{N} D_{x} D_m \Phi(t,z,m_{\bx}^N,x^i) + \frac{1}{N^2} D_{mm} \Phi(t,z,m_{\bx}^N,x^i,x^i), 
\end{align*}
we deduce that 
\begin{align} \label{finitedimsub}
\displaystyle  - &\partial_t \Phi(t^N,z^N,m_{\bx^N}^N) - a_0 \Delta_z \Phi(t^N,z^N,m_{\bx^N}^N) -\int_{\R^d} \Delta_x \frac{\delta \Phi}{\delta m}(t^N,z^N,m^N_{\bx^N},x) m(dx) \nonumber  \\ 
&\qquad   - \int_{\R^d} \ov{H}(x, z^N,  D_x \frac{\delta \Phi}{\delta m}(t^N,z^N,m_{\bx^N}^N,x), m_{\bx^N}^N)  m^N_{\bx^N}(dx)  \vspace{.1cm} \nonumber  \\
&= - \partial_t \Phi^N(t^N,z^N,\bx^N) - \sum_{ i = 1}^N \Delta_{x^i} \ov{V}^N(t^N,z^N,\bx^N) - a_0 \Delta_z \ov{V}^N(t^N,z^N,\bx^N) \nonumber \\
&\qquad  + \frac{1}{N} \sum_{i = 1}^N \ov{H}(x^{N,i},z,ND_{x^i} \ov{V}^N(t^N,\bx^N), m_{\bx}^N) 
 + \frac{1}{N^2} \sum_{i = 1}^N \tr(D_{mm} \Phi(t^N,z^N,m_{\bx}^N,x^i,x^i)) \nonumber \\ 
 &\leq C/N,
\end{align}
where the last line uses the fact that $D_{mm} \Phi$ is bounded. We send $N \to \infty$, and combine \eqref{approxoptimizer} with \eqref{finitedimsub} to see that 
\begin{align*}
\displaystyle  &- \partial_t \Phi(t_0,z_0,m_0) - a_0 \Delta_z \Phi(t_0,z_0,m_0) - \int_{\R^d} \Delta_x \frac{\delta \Phi}{\delta m}(t_0,z_0,m_0,x) m_0(dx)  \\
&\qquad \qquad +\int_{\R^d} \ov{H}(x, z_0, D_x \frac{\delta \Phi}{\delta m}(t_0,z_0,m_0,x), m_0)  m_0(dx) \leq 0.
\end{align*}
This shows that $U$ is a viscosity subsolution of \eqref{hjb}. The proof that $U$ is a supersolution is almost identical.
\end{proof}

The next goal of this section is to show that under our main assumptions, there exists a limit point of the sequence $(V^N)_{N \in \N}$, thereby completing the proof of existence.
%The following proposition will provide the necessary compactness.

%We can now complete the proof of existence. 

\begin{proof}[Proof of Theorem \ref{thm.exist}]
    Using Lemma \ref{lem:VnLip}, we see that we can find a constant $C$ and a sequence of functions $\widetilde{V}^N : [0,T] \times \cP(\T^d) \to \R$ such that 
    \begin{align*}
    &\widetilde{V}^N(t,m_{\bx}^N) = V^N(t,\bx) \quad \text{for } (t,\bx) \in [0,T] \times (\T^d)^N, \text{ and } \\
        &|\widetilde{V}^N(t,m) - \widetilde{V}^N(t',m')| \leq C\big(|t-t'|^{1/2} + \bd_1(m,m') \big) \text{ for $t,t' \in [0,T]$, $m,m' \in \cP(\T^d)$.}
    \end{align*}
    Applying the Arzel\'a-Ascoli compactness criterion, we deduce the existence of a function $U : [0,T] \times \cP(\T^d) \to \R$ such that $\widetilde{V}^N \to U$ uniformly. In particular, it follows that $U$ is a limit point of the sequence $(V^N)_{N \in \N}$, which satisfies the estimate \eqref{d1regintro}. Then Proposition \ref{prop.viscosity} shows that this limit point is a viscosity solution.

%    \red{Do we need to write more? e.g. Lemma \ref{lem:VnLip} implies that $\hat V^N$ is $d_1$-Lipschitz and so has limit points by Arzela-Ascoli, which then must be limit points of $V^N$ JJ: is the above okay?}

\end{proof}

In the following section, it will also be important to know that if $H$ and $G$ satisfy additional smoothness condition, then limit points of $(V^N)_{N \in \N}$ are Lipschitz with respect to weaker norms. For this, we will need the following mollification procedure. We note that as in Lemma \ref{lem.c2moll}, the construction is again based on the strategy of \cite{cossomartini}, but we nevertheless provide a full proof in order to verify the fact that $H^n$ and $G^n$ satisfy the regularity requirements of Assumption \ref{assump.d1} uniformly in $n$.

\begin{lem} \label{lem.mollification}
    Suppose that Assumption \ref{assump.d1} holds. Then there exist a sequence $(H^{n}, G^{n})_{n \in \N}$ such that 
    \begin{enumerate}
        \item $H^{n}  : \T^d \times \R^d \times \cP(\T^d) \to \R$, $G^{n} : \cP(\T^d) \to \R$ admits continuous derivatives of all orders,
        \item $H^{n}$, $G^{n}$ satisfy the regularity conditions appearing in Assumption \ref{assump.d1}, uniformly in $n$,
        \item $H^{n} \xrightarrow{n \to \infty} H$ locally uniformly, and $G^{n} \xrightarrow{n \to \infty} G$ uniformly. 
    \end{enumerate}
\end{lem}

\begin{proof}
    We explain the mollification procedure only for $H$, since it will be clear how to use the same construction to mollify $G$. We will follow the same strategy as in the proof of Lemma \ref{lem.c2moll}. First, for each $n \in \N$, we define a projection $\ov{H}^n$ of $H$ as follows:
    \begin{align*}
        \ov{H}^n : \T^d \times \R^d \times (\T^d)^N \to \R, \quad \ov{H}^n(x,p,\by) = H(x,p,m_{\by}^N). 
    \end{align*}
    Now let $(\phi_{\epsilon})_{\epsilon > 0}$ be a standard mollifier on $\R^d$ with $\text{supp}(\phi_{\epsilon}) \subset B_{\epsilon} \subset \R^d$, and define
    \begin{align*}
        \ov{H}^{n,\epsilon}(x,p,\by) = \int_{\R^d} \int_{\R^d} \int_{(\R^d)^n} \ov{H}^n(x', p', \by') \rho_{\epsilon}(x-x') \rho_{\epsilon}(p - p') \prod_{i = 1}^n \rho_{\epsilon}(y^i - y^{'i}) dx' dp' d\by'.
    \end{align*}
    Finally, we define $H^n : \T^d \times \R^d \times \cP(\T^d) \to \R$ via
    \begin{align*}
        H^n(x,p,m) = \int_{(\T^d)^n} \ov{H}^{n,1/n}(x,p,\by) m^{\otimes n}(d\by).
    \end{align*}
    We now verify that $H^n$ has the desired properties. First, it is easy to see that $\ov{H}^{n,\epsilon}$ is smooth for each $n \in \N$ and $\epsilon > 0$, from which it follows that $H^n$ is smooth. 
    
    Next, we show that $H^n$ satisfies \eqref{assump.H} uniformly in $n$. Note that because $H$ satisfies the condition \eqref{assump.H}, the functions $\ov{H}^n$ satisfy 
    \begin{align*}
        |\ov{H}^n(x,p,\by) - \ov{H}^n(x',p', \by')| \leq C\big(1 + |p| + |p'|) \Big(|x - x'| + |p - p'| + \frac{1}{n} \sum_{i = 1}^n |y^i - y^{'i}| \Big).
    \end{align*}
    In particular, increasing $C$ if necessary, we find that for all $0 < \epsilon \leq 1$, and each $n \in \N$,
    \begin{align} \label{derivativeest}
        |D_{x} \ov{H}^{n,\epsilon} (x,p,\by)| + |D_{p} \ov{H}^{n,\epsilon} (x,p,\by)| + n\max_{i = 1,...,n} |D_{y^i} \ov{H}^{n,\epsilon}(x,p,\by)| \leq C(1 + |p|).
    \end{align}
    But we can compute explicitly 
    \begin{align*}
        D_x H^n(x,p,m) &= \int_{(\T^d)^n} D_x\ov{H}^{n,1/n}(x,p,\by) m^{\otimes n}(d\by), \\
        D_p H^n(x,p,m) &= \int_{(\T^d)^n} D_p\ov{H}^{n,1/n}(x,p,\by) m^{\otimes n}(d\by), \\
        D_m H^n(x,p,m,z) &= \sum_{i = 1}^n \int_{(\T^d)^{n-1}} D_{y^i} H^{n,1/n}(x,p,(\by^{-i},z)) m^{\otimes (n-1)}(d\by^{-i}). 
    \end{align*}
%    where we use the notation 
%    \begin{align*}
%        \by^{-i} = (y^1,...,y^{i-1},y^{i+1},...,y^N) \in (\T^d)^{N-1}, \quad (\by^{-i},x) = (y^1,...,y^{i-1},x,y^{i+1},...,y^N) \in (\T^d)^N.
%    \end{align*}
    When combined with \eqref{derivativeest}, this shows that 
    \begin{align*}
        |D_x H^n(x,p,m)| + |D_p H^n(x,p,m)| + |D_m H^n(x,p,m)| \leq C( 1 + |p|), 
    \end{align*}
    which easily implies that the functions $H^n$ satisfy \eqref{assump.H} uniformly in $n$. 
    
    Finally, we show that $H^n \to H$ locally uniformly. Since $\text{supp}(\phi_{\epsilon}) \subset B_{\epsilon}$, we can estimate 
    \begin{align*}
        |H^{n,1/n}(x,p,\by) - \ov{H}^n(x,p,\by)| \leq \frac{C}{n}(1 + |p|), 
    \end{align*}
    and thus 
    \begin{align*}
        |H^n(x,p,m) - H(x,p,m)| &\leq  \int_{(\T^d)^n} |H^{n,1/n}(x,p,\by) - H(x,p,m)| m^{\otimes n}(d\by)
        \\
        &\leq \frac{C}{n}(1 + |p|) + \int_{(\T^d)^n} |\ov{H}^n(x,p,\by) - H(x,p,m)| m^{\otimes n}(d\by) \\
        &= \frac{C}{n}(1 + |p|) + \int_{(\T^d)^n} |H(x,p,m_{\by}^n) - H(x,p,m)| m^{\otimes n}(d\by) \\
        &\leq \frac{C}{n}(1 + |p|) + C (1 + |p|) \int_{(\T^d)^n} \bd_1(m_{\by}^n, m) m^{\otimes n}(d\by) 
        \\
        &\leq C(1 + |p|)\big(N^{-1} + N^{-1/d} \big),
    \end{align*}
    which completes the proof. 
\end{proof}

We can now complete the proof of Theorem \ref{thm.exist}. Actually, we will prove the following more general statement. 

\begin{prop} \label{prop.existsmooth}
    Suppose that Assumption \ref{assump.d1} holds, and in addition \eqref{glipk} and \eqref{hlipk} hold for some $k \in \N$. Then there exists a solution $U$ of \eqref{hjb} which satisfies the regularity estimate in \eqref{regintro}.
\end{prop}

\begin{proof}
    Assume for the moment that $H$ and $G$ are smooth, i.e. admit bounded derivatives of all orders. Let $U$ denote any limit point of $(V^N)_{N \in \N}$. Our goal is now to prove that $U$ satisfies \eqref{regintro}. Without loss of generality, we assume that 
    \begin{align} \label{wlog}
        \sup_{(t,\bx)} |V^N(t,m) - U(t,m_{\bx}^N) | \to 0.
    \end{align}
    Because $H$ and $G$ are smooth, so is $V^N$, and so the function 
    \begin{align*}
        \hat{V}^N : [0,T] \times \cP(\T^d) \to \R, \quad \hat{V}^N(t,m) = \int_{(\T^d)^N} V^N(t,\by) m^{\otimes N}(d\by)
    \end{align*}
    is smooth. Its linear derivative is given by 
    \begin{align*}
        \frac{\delta \hat{V}^N}{\delta m}(t,m,x) = \sum_{i = 1}^N \int_{(\T^d)^{N-1}} V^N(t,(\by^{-i},x)) m^{\otimes (N-1)}(d\by^{-i}). 
    \end{align*}
%    where we use the notation 
%    \begin{align*}
%        \by^{-i} = (y^1,...,y^{i-1},y^{i+1},...,y^N) \in (\T^d)^{N-1}, \quad (\by^{-i},x) = (y^1,...,y^{i-1},x,y^{i+1},...,y^N) \in (\T^d)^N.
%    \end{align*}
    In particular, using Theorem \ref{thm:VNestimates}, we have 
    \begin{align*}
       \sup_{t,m} \|\frac{\delta \hat{V}^N}{\delta m}(t,m,\cdot)\|_{C^k} \leq C\sum_{1 \leq j \leq k} \sum_{i = 1}^N \|D_{x^i}^k V^N \|_{\infty} \leq C_0 + C_1/\sqrt{N}, 
    \end{align*}
    where $C_0$ depends only on the constants $C_H$, $C_{H,k}$, $C_{G,k}$ appearing in Assumption \ref{assump.d1} and \eqref{R}, while $C_1$ is independent of $N$. Using the definition of the linear derivative,
    \begin{align} \label{vnhatest}
        |\hat{V}^N(t,m) - \hat{V}^N(t,m')| &= \int_0^1 \int_{\T^d} \frac{\delta V^N}{\delta m}(t,sm + (1-s)m', x) (m - m')(dx) \nonumber \\
        &\leq \sup_{s,m} \|\frac{\delta \hat{V}^N}{\delta m}(s,m,\cdot)\|_{C^k} \|m - m'\|_{C^{-k}} \leq \big(C_0 + C_1/N \big) \|m - m'\|_{C^{-k}},  
    \end{align}
     where $C_0$ depends only on $C_H$, $C_{H,k}$, $C_{G,k}$. Next, using the same argument as in the proof of Proposition 6.2 in \cite{DDJ_23}, we have
    \begin{align*}
        \sup_{(t,\bx)} |\hat{V}^N(t,m_{\bx}^N) - V^N(t,\bx)| \to 0, 
    \end{align*}
    which together with \eqref{wlog} and the fact that the functions $\hat{V}^N$ and $U$ satisfy the estimate \eqref{d1regintro} uniformly, implies that
    \begin{align*}
        \sup_{t,m} |U(t,m) - \hat{V}^N(t,m)| \to 0. 
    \end{align*}
    In particular, we can pass to the limit in \eqref{vnhatest} to infer that $U$ satisfies
    \begin{align} \label{lipck}
        |U(t,m) - U(t,m')| \leq C_0 \|m - m'\|_{C^{-k}}, 
    \end{align}
    with $C_0$ depending only on $C_H$, $C_{H,k}$, $C_{G,k}$.
    
    At this point, we have shown that if $H$ and $G$ satisfy \eqref{hlipk} and \eqref{glipk} in addition to Assumption \ref{assump.d1}, and in addition $H$ and $G$ are smooth, then there exists a solution $U$ which satisfies \eqref{regintro}, with constant $C_0$ depending only on $C_H$, $C_{k,R}$. Moreover, when $k \geq 2$, Theorem \ref{thm:VNestimates} implies
    \begin{align*}
        | \partial_t V^N| = \bigg| \sum_{i = 1}^N \Delta_{x^i} V^N + a_0 \sum_{i,j = 1}^N \tr\big(D_{x^ix^j} V^N \big) - \frac{1}{N} \sum_{i = 1}^N H(x^i, ND_{x^i} V^N, m_{\bx}^N) \bigg| \leq C_0' + C_1'/\sqrt{N}, 
    \end{align*}
    where $C_0'$ depends only on $C_H$, $C_{H,k}$, $C_{G,k}$. Passing this estimate to the limit, we find that if $k \geq 2$, then the limit point $U$ we have constructed is Lipschitz in time, with constant depending only on $C_H$, $C_{H,k}$, $C_{G,k}$. 
    
    To complete the proof we must remove the assumption that $H$ and $G$ are smooth. For this, we apply Lemma \ref{lem.mollification} to produce $(H^{\epsilon}, G^{\epsilon})_{\epsilon > 0}$ satisfying conditions (1)-(3) in the statement of Lemma \ref{lem.mollification}. Next, let $U^{\epsilon}$ denote a solution to \eqref{hjb} with data $G^{\epsilon}, H^{\epsilon}$ which satisfies the estimate \eqref{regintro} with a constant $C$ independent of $\epsilon$. Then by Arzela-Ascoli, we can find a sequence $(\epsilon_j)_{j \in \N}$ and a function $U : [0,T] \times \cP(\T^d) \to \R$ such that $U^{\epsilon_j} \to U$ uniformly as $j \to \infty$. It is straightforward to check that $U$ is a viscosity solution of \eqref{hjb} satisfying the estimate \eqref{regintro}. 
\end{proof}

\section{Comparison and convergence} \label{sec.comparison}

The goal of this section is to prove the comparison principle, Theorem \ref{thm.comparison}, and explain how it implies the convergence result, Theorem \ref{thm.convergence}. We fix throughout this section a $k \in \N$ with 
\begin{align*}
    k > d/2 + 2.
\end{align*}
We start with some preliminary notation and reminders about the Sobolev spaces $H^{k}$ and $H^{-k}.$ First, we define
\[
    H^k := \bigl \{ \varphi \in L^2; \mbox{  for all  }j \in \N_0^d, \mbox{ with } |j| \leq k, D^{j} \varphi \mbox{ exists in the weak sense and belongs to } L^2 \bigr \}
\]
It becomes a Hilbert space when endowed with the inner product
$$ \langle \varphi_1, \varphi_2 \rangle_{k} :=  \sum_{ |j| \leq k}  \langle D^{j} \varphi_1, D^{j} \varphi_2 \rangle_{L^2(\T^d)}. $$
We denote by $\norm{\cdot}_{k}$ the associated norm. We define $H^{-k}$ as the dual of $H^k$, i.e. the set of bounded linear functionals over $H^k$, and we denote
$$ \norm{ q }_{-k} := \sup_{\varphi \in H^k,  \norm{\varphi}_{k} \leq 1} q(\varphi) $$
the associated norm. For any $q \in H^{-k}$ we denote by $q^{*}$ its dual element, i.e. the unique element of $H^k$ satisfying
$$q(\varphi) = \langle q^{*}, \varphi \rangle_{k}, \quad \mbox{for all } \varphi \in H^k.$$
The map $q \mapsto q^*$ is an isometry from $H^{-k}$ into $H^k$, if the inner product in $H^{-k}$ is defined by $\langle \mu, \nu\rangle_{-k} = \langle \mu^*, \nu^* \rangle_k$, which also coincides with the definition of the norm $\norm{\cdot}_{-k}$.

Recall that, for $k > d/2 + 2$, $H^k$ continuously embeds in $C^2(\T^d)$. In particular, $\mathcal{P}(\T^d)$ can be seen as a compact subset of $H^{-k}$ via the identification $\mu(\varphi) = \int_{\T^d} \varphi d\mu$, for all $\varphi \in H^k$ and all $\mu \in \mathcal{P}(\T^d)$. Finally from \cite{DDJ_23} Lemma 2.11, or a simple computation, (the restriction of ) any function $\Phi \in C^1(H^{-k})$ (in the classical Hilbertian sense) belongs to $\cC^1(\mathcal{P}(\T^d))$ and
$$ \frac{\delta \Phi}{\delta m} (m,x) = D_{-k} \Phi(m)(x) - \int_{\T^d} (D_{-k} \Phi(m) \bigr)(y)dy,$$
where $D_{-k} \Phi \in H^k$ is the (Frechet) differential of $\Phi.$ In particular, by our choice of convention \eqref{normalizationconvention}, for any $\nu \in \mathcal{P}(\T^d)$, $\mu \mapsto \frac{1}{2} \norm{\mu -\nu}_{-k}^2$ belongs to $C^1(\mathcal{P}(\T^d))$ with
$$ \frac{\delta }{\delta m}( \mu \mapsto \frac{1}{2}\norm{\mu - \nu}^2_{-k} )= (\mu - \nu)^{*}.$$
\begin{rmk}
The Hilbert spaces $H^{k}$ and $H^{-k}$ introduced above can be explicitly defined in terms of the Fourier coefficients of the corresponding elements, as in \cite{DDJ_23, Soner2022}. Our choice of inner product and norm are slightly different than in those works, though equivalent and with the same regularity properties. In particular, the map
\[
    \mathcal P(\T^d) \ni \mu \mapsto \mu^* \in H^k(\T^d)
\]
is a smoothing operator that can be explicitly described by the inverse of a certain differential operator, although this particular characterization is not relevant to the results in this section.

Let us also note that the idea of measuring the distance between $\mu,\nu \in \mcl P(\T^d)$ by integrating against special test functions is introduced also in \cite{Gretton_2012}, where the relevant metrics are known as maximum mean discrepancy and applied to problems in machine learning. 
\end{rmk}

It will be convenient to introduce a notion of semi-jets as follows.

\begin{defn} \label{def.semijet}
    For $U : [0,T] \times \T^d \times \mathcal{P}(\T^d) \rightarrow \R  \cup \{ - \infty \} $  an upper semi-continuous function, and $(t, z, \mu) \in [0,T] \times \T^d \times \mathcal{P}(\T^d)$ such that $\Phi( t, z, \mu) >-\infty$ we define $J^{2,+} V(t, z , \mu)$ as the set of elements 
    $$(r,p,X,\varphi) \in \R \times \R^d \times S^d(\R) \times C^2(\T^d)$$ such that there exists a $\cpart$ function $\Phi : [0,T] \times \T^d \times \mathcal{P}(\T^d)$ satisfying
%    \begin{enumerate}
%        \item $u - \Phi$ has a local maximum at $(t,z,\mu)$, 
 %       \item $\bigl( \partial_t \Phi, D_z \Phi,  \frac{\delta \Phi}{\delta m} \bigr) ( t , z , \mu ) = (r,p,\varphi),$ and 
  %      \item $D^2_{zz} \Phi( t, z, \mu ) \leq X$
 %   \end{enumerate}
\begin{equation}
\left \{
\begin{array}{ll}
       \displaystyle V - \Phi \mbox{ has a local maximum at }(t,z,\mu), \vspace{.1cm}  \\
       \displaystyle \bigl( \partial_t \Phi, D_z \Phi,  \frac{\delta \Phi} {\delta m} \bigr) ( t , z , \mu ) = (r,p,\varphi), \vspace{.1cm} \\ 
        \displaystyle D^2_{zz} \Phi( t, z, \mu ) \leq X.
\end{array}
\right.
\end{equation}

   % $$u - \Phi \mbox{ has a local maximum at } (t , z , \mu)$$
%as well as
%\begin{equation}
%\left \{  
%\begin{array}{ll}
%\bigl( \partial_t \Phi, D_z \Phi, D_{-s} \Phi \bigr) ( t , z , \mu ) = (r,p,\varphi),  \vspace{.1cm} \\
%D^2_{zz} \Phi( t, z, \mu ) \leq X
%\end{array}
%\right.
%\end{equation}
%    $$u(t,x,\mu) \leq u(\bar{t}, \bar{x} , \bar{\mu}) + r(t- \bar{t}) + p.(x- \bar{x}) + \varphi \cdot (\mu - \bar{\mu}) +   \frac{1}{2}X(x-\bar{x}).(x-\bar{x}) + o \bigl( |t-\bar{t}| + | \mu - \bar{\mu} | + |x- \bar{x}|^2 \bigr).$$
%\textcolor{red}{This is not the right definition. I should define everything in terms of smooth test functions touching from above}
Similarly, for a lower-semicontinuous function $V: [0,T] \times \T^d \times \cP(\T^d) \rightarrow \R  \cup \{ + \infty \} $ and some $(t, z, \mu) \in [0,T] \times \T^d \times \cP(\T^d)$ such that $V( t, z, \mu) <+\infty$ we define $J^{2,-}V(t, z, \mu) = -J^{2,+} (-V)(t, z, \mu).$

We also say that $(r,p,X,\varphi) \in \R \times \R^d \times S^d(\R) \times C^2(\T^d)$ belongs to $\bar{J}^{2,+}V(t , z , \mu)$ if there are sequences $(t_n, z_n, \mu_n)_{ n \geq 0}$ and $(r_n,p_n,X_n, \varphi_n)_{n \geq 0}$ such that 
\begin{equation} 
(r_n,p_n,X_n, \varphi_n) \in J^{2,+}u(t_n,z_n,\mu_n), \forall n \geq 0, \mbox{ and }
\end{equation}
\begin{equation}
    \left \{
    \begin{array}{ll}
    (t_n,z_n, \mu_n) \to (t,x,\mu), \vspace{.1cm} \\ 
   V(t_n,z_n,\mu_n) \to V(t,z,\mu),  \vspace{.1cm}\\
   (r_n,p_n, X_n ,\varphi_n) \to (r,p,X,\varphi), \\
 %  (r_n,p_n,X_n, \varphi_n) \in J^{2,+}u(t_n,z_n,\mu_n). 
   \end{array}
   \right.
\end{equation}
as $n \rightarrow +\infty$, where convergence for $(\varphi_n)_{n \geq 0}$ holds in $C^2(\T^d)$. We define $\bar{J}^{2,-}V$ analogously.

%\begin{enumerate}
%   \item $(t_n,z_n, \mu_n) \to (t,x,\mu)$, 
%   \item $u(t_n,z_n,\mu_n) \to u(t,z,\mu)$
%   \item $(r_n,p_n, X_n ,\varphi_n) \to (r,p,X,\varphi)$,
%   \item $(r_n,p_n,X_n, \varphi_n) \in J^{2,+}u(t_n,z_n,\mu_n)$.
%\end{enumerate}
\end{defn}

It is straightforward to check that the notion of viscosity solution of Definition \eqref{def.viscosityz} can be equivalently formulated in terms of these sub/super-jets. We record this fact as a lemma.

\begin{lem}
    Assume Assumption \ref{assump.d1}. An upper-semi-continuous function $V : [0,T] \times \T^d \times \mathcal{P}(\T^d) \rightarrow \R \cup \{- \infty\}$ is a viscosity sub-solution of \eqref{hjbz} in the sense of Definition \ref{def.viscosityz} if and only if, $V(T, \cdot) \leq G$ and, for any $(t,z,\mu) \in [0,T] \times \mathcal{P}(\T^d) \times \cP(\T^d)$ such that $V(t,z,\mu) >-\infty$ and any $(r,p,X,\varphi) \in \bar{J}^{2,+}(t,z,\mu)$ it holds
\begin{equation}
-r - a_0 \tr(X) - \int_{\T^d} \Delta_x \varphi(x) d\mu(x) + \int_{\T^d} \ov{H}(x, z, D_x\varphi(x), \mu) d\mu(x) \leq 0.
\label{equivalencedefviscositysolutions}
\end{equation}
A similar statement holds for super-solutions.
%Similarly, a lower-semi-continuous function $v: [0,T] \times \T^d \times \mathcal{P}(\T^d) \rightarrow \R \cup \{+ \infty \}$ is a viscosity super-solution of \eqref{hjbz} in the sense of Definition \ref{def.viscosityz} if and only if, $v(T,\cdot) \geq G$ and, for any $(t,z,\mu) \in [0,T] \times \T^d \times \mathcal{P}(\T^d)$ such that $v(t,z,\mu) <+\infty$ and any $(r,p,X,\varphi) \in \bar{J}^{2,-}(t,z,\mu)$ it holds
%\begin{equation}
%-r - a_0 \tr(X) - \int_{\T^d} \Delta_x \varphi(x) d\mu(x) + \int_{\T^d} H(x+z, D_x\varphi(x), \mu) d\mu(x) \geq 0.
%\label{equivalencedefviscositysolutionssuper}
%\end{equation}
\end{lem}

%\red{JJ: Is it clear that the opposite implication holds? I.e., a viscosity solution defined in terms of these semijets is also a viscosity solution in the sense above? If yes, I think we should formalize this in a Lemma, i.e. the above paragraph should be replaced with a Lemma stating precisely the equivalence.}

The next proposition establishes the comparison principle under the additional assumption that either $U$ or $V$ satisfies the regularity condition 
\begin{align} 
    &|\Phi(t,z,m) - \Phi(t,z,m')| \leq C \|m - m'\|_{-k}, \notag\\
    & \text{for some $C > 0$, for all $t \in [0,T], z \in \T^d$, $m,m' \in \cP(\T^d)$}. \label{hslip}
\end{align}

\begin{prop} \label{prop.partialcomparison}
    Assume that $U,V : [0,T] \times \T^d \times \mathcal{P}(\T^d) \rightarrow \R$ are a viscosity sub-solution and a viscosity super-solution to \eqref{hjbz}, respectively. Assume moreover that either $U$ or $V$ satisfies the Lipschitz condition \eqref{hslip}. Then, 
    \begin{align*}
        U(t,z,m) \leq V(t,z,m)
    \end{align*}
    for all $(t,z,m) \in [0,T] \times \T^d \times \mathcal{P}(\T^d)$.
\end{prop}

%\red{JJ: I have included the boundary condition in the definition of viscosity sub/super-solution, and so written the statement of this Proposition slightly differently.}

%\textcolor{red}{Maybe we need to specify the topology when referring to upper and lower semi-continuity (I'm not sure these are stable notions wrt topologically equivalent metrics). }

%\red{JJ: they are stable, since if two metrics are topologically equivalent then they have the same convergent subsequences, and lower/upper semi-continuity can be defined only in terms of convergence subsequences.}

\begin{proof}
We suppose that
\begin{align} \label{stc}
    \sup_{t,z,m} \Big( U(t,z,m) - V(t,z,m) \Big) > 0,
\end{align}
and then derive a contradiction through several steps. \newline \newline 
\noindent \textit{Step 1: Basic penalizations and related estimates.} For $\gamma > 0$, define
	\[
		U^\gamma(t,z,\mu) := U(t,z,\mu) - \gamma(T-t) - \gamma\left( \frac{1}{t} - \frac{1}{T}\right).
	\]
	It is straightforward to check that 
	\begin{equation}\label{penU}
		\left\{
		\begin{split}
		&U^\gamma(T,\cdot,\cdot) = U(T,\cdot, \cdot), \quad U^\gamma \le U, \quad \lim_{t \to 0^+} U^\gamma = -\infty, \quad \text{and}\\
		&\text{$U^\gamma$ is a subsolution of \eqref{hjbz} with right-hand side $-\gamma$.}
		\end{split}
		\right.
	\end{equation}
	By \eqref{stc} and compactness of $[0,T] \times \T^d \times \mathcal{P}(\T^d)$, we can choose $\gamma > 0$ small enough that
	\begin{equation}\label{firstmax}
		(t,z,m) \mapsto U^{\gamma}(t,z,m) - V(t,z,m)
	\end{equation}
	attains a strictly positive maximum value $\omega > 0$ at some $(t_*,z_*,m_*)$.
 %\red{Is it clear that a maximum is attained in $\mcl P$? SD: the maps are upper semi-continuous/lower semi-continuous by definition and $[0,T] \times \T^d \times \mathcal{P}(\T^d)$ is compact so I think it is OK not to say more ?}
 In view of the penalizing property \eqref{penU} and the fact that $U^{\gamma}(T,\cdot,\cdot) = U(T,\cdot, \cdot) \leq V(T,\cdot, \cdot)$, we have $t_* \in (0,T)$. For $\epsilon >0$ we define the functional $\Phi^{\epsilon} : [0,T]^2 \times (\T^d)^2 \times \big(\mathcal{P}(\T^d)\big)^2 \to \R$ by
 $$\Phi^{\epsilon}(\theta) = U^{\gamma}(s,y, \mu) - V(t,z,\nu) - \frac{1}{2\epsilon} \norm{\mu - \nu}_{-k}^2 - \frac{1}{2 \epsilon} |y-z|^2 - \frac{1}{2\epsilon} |t-s|^2$$
 where $\theta = (s,t,y,z,\mu,\nu)$ is a generic element of $[0,T]^2 \times (\T^d)^2 \times \big(\mathcal{P}(\T^d) \big)^2.$
% \red{I know this is purely cosmetic, but I find it slightly ugly to use $z_1$ and $z_2$ and $t$,$s$ (rather than $t_1$,$t_2$). Could we use $z$ and $w$?  SD: I've changed $z_1 -> y$ and $z_2->z$.}
 Moreover, we define 
 $$\displaystyle M^{\epsilon} := \sup_{\theta \in [0,T]^2 \times (\T^d)^2 \times \bigl(\mathcal{P}(\T^d) \bigr)^2} \Phi^{\epsilon}(\theta)$$
 and we denote by $ \bar{\theta} := (\bar{s} , \bar{t} , \bar{y} , \bar{z} , \bar{\mu}, \bar{\nu})$ a maximizer of $\Phi^{\epsilon}$. The map $\epsilon \rightarrow M^{\epsilon}$ is non-increasing and lower bounded by $\omega$ and therefore converges as $\epsilon \rightarrow 0^+$. In view of the inequality
$$M^{\epsilon} \leq M^{\epsilon/2} - \frac{1}{2\epsilon} \norm{\bar{\mu} - \bar{\nu}}_{-k}^2 - \frac{1}{2 \epsilon} |\bar{y}-\bar{z}|^2 - \frac{1}{2\epsilon} |\bar{s}-\bar{t}|^2 $$
and the fact that $\lim_{\epsilon \rightarrow 0^+} M^{\epsilon} - M^{\epsilon/2} = 0$ we conclude that 
\begin{equation}
\lim_{\epsilon \rightarrow 0^+} \frac{1}{2\epsilon} \norm{\bar{\mu} - \bar{\nu}}_{-k}^2 + \frac{1}{2 \epsilon} |\bar{y}-\bar{z}|^2 + \frac{1}{2\epsilon} |\bar{s}-\bar{t}|^2=0 
\label{Penalizationgoesto0}
\end{equation}
and every (weak) limit point of $(\bar{s}, \bar{y} , \bar{\mu})$ and $(\bar{t} , \bar{z}, \bar{\nu})$ must coincide and be a maximum of $U^{\gamma}-V$. In particular $\bar{s}, \bar{t}$ belong to $(0,T)$ for $\epsilon$ small enough. 
%This also means that
%$$\lim_{\epsilon \rightarrow 0^+} M^{\epsilon} = \omega$$
%and we can assume that $\epsilon$ is small enough so that $M^{\epsilon} \geq \frac{\omega}{2}.$ \textcolor{red}{SD: where do we use it in the rest of the proof ?} 
\newline \newline 
\noindent \textit{Step 2: Estimates from regularity.} Assume without loss of generality that $U$ is Lipschitz continuous in $\mu$ with respect to $\norm{\cdot}_{-k}$ with Lipschitz constant $L > 0$, uniformly in $(t,z)$. Then, comparing $M^{\epsilon}$ with $\Phi^{\epsilon}(\bar{s} , \bar{t} , \bar{y} , \bar{z} , \bar{\nu} , \bar{\nu})$, we have
$$ \frac{1}{2 \epsilon} \norm{\bar{\mu} - \bar{\nu}}_{-k}^2 \leq U^{\gamma}( \bar{s} , \bar{y}, \bar{\mu} ) - U^{\gamma} ( \bar{s}, \bar{y} ,\bar{\nu}) \leq L \norm{ \bar{\mu} - \bar{\nu}}_{-k}   $$
which leads to
\begin{equation} 
\frac{1}{\epsilon} \norm{\bar{\mu} - \bar{\nu}}_{-k} \leq 2  L.
\label{estimatefromregularity}
\end{equation}
%Assuming instead that $V$ is Lipschitz with respect to $\norm{\cdot}_{-k}$ leads to the same bound.
Since $k \geq d/2$, the classical Sobolev embedding implies also that
\begin{equation} 
\frac{1}{\epsilon} \norm{ (\bar{\mu} - \bar{\nu})^{*}}_{\infty} \leq \frac{C}{\epsilon} \norm{ (\bar{\mu} - \bar{\nu})^{*} }_k = \frac{C}{\epsilon} \norm{ \bar{\mu} - \bar{\nu}}_{-k} \leq 2CL,
\label{estimateLinftyfromregularity}
\end{equation}
for some $C >0$ depending on dimension only.

At this stage, we need the following key lemma, which is a sort of ``parameterized version" of the classical Ishii's lemma in finite dimensions. 

\begin{lem}
    Let $\bar{\theta} = (\bar{s} , \bar{t} , \bar{y} , \bar{z} , \bar{\mu}, \bar{\nu})$ be a local maximum of $\Phi^{\epsilon}$. For every $\kappa >0$ we can find two matrices $X,Y \in S^d(\R)$ satisfying the matrix inequality
\begin{equation} 
-\bigl(\frac{1}{\kappa} + \frac{1}{\epsilon} \bigr) I_{2d} \leq \begin{pmatrix} X & 0 \\ 0 & -Y \end{pmatrix} \leq \frac{1 + 2 \kappa}{2 \epsilon} \begin{pmatrix} I_d & -I_d \\ -I_d & I_d \end{pmatrix} 
\label{MatrixInequality9oct2023}
\end{equation}
and such that
\begin{equation}
    \begin{array}{ll}
\Bigl( \frac{1}{\epsilon}(\bar{s} - \bar{t}), \frac{1}{\epsilon}(\bar{y}- \bar{z}), X, \frac{1}{\epsilon} ( \bar{\mu} - \bar{\nu})^{*} \Bigr) \in \bar{J}^{2,+} U^{\gamma} \bigl(\bar{s} , \bar{y}, \bar{\mu} \bigr), \\
\Bigl( \frac{1}{\epsilon}(\bar{s} - \bar{t}), \frac{1}{\epsilon}(\bar{y}- \bar{z}), Y, \frac{1}{\epsilon} ( \bar{\mu} - \bar{\nu})^{*} \Bigr) \in \bar{J}^{2,-} V \bigl(\bar{t} , \bar{z}, \bar{\nu} \bigr).
    \end{array}
\label{IshiisInclusion21oct}
\end{equation}

%\textcolor{red}{SD: do we need some compactness for the time derivatives ? (cf proof of the parabolic version of Ishii's lemma). Check condition (27) in Crandall/Ishii's theorem 9. Edit: I think it is OK to use the standard form of Ishii's lemma (not the parabolic one) since we also double variables in time.}

%as well as a sequence $(\phi_n)_{n \geq 0} : [0,T]^2  \times (\T^d)^2 \times (\mathcal{P}(\T^d))^2$ of smooth test functions and a sequence $(\theta_n) = (t_n,s_n,z^1_n,z^2_n,\mu_n,\nu_n)$ satisfying
%\begin{equation}
%    \left \{
%    \begin{array}{ll}
%\theta^n  \mbox{ is a local maximum of }  U^{\gamma} -V - \phi_n \\
%\theta^n \rightarrow \bar{\theta}, \mbox{ as } n \rightarrow +\infty \\
% \bigl( \partial_t \phi_n, \partial_s \phi_n, D_{z_1} \phi_n, D_{z_2}\phi_n, D^2_{z_1,z_1} \phi_n , D^2_{z_2,z_2}\phi_n, \frac{\delta \phi_n}{\delta \mu}, \frac{\delta \phi_n}{\delta \nu} \bigr)  (\theta^n)\\
% \rightarrow_{ n \rightarrow +\infty} \frac{1}{\epsilon} \Bigl( \bar{t} - \bar{s} , \bar{s} - \bar{t} , \bar{z}_1- \bar{z}_2, \bar{z_2} - \bar{z_1}, \epsilon X , - \epsilon Y , (\bar{\mu} - \bar{\nu})^{*}, (\bar{\nu} - \bar{\mu})^{*} \Bigr)
%    \end{array}
%    \right.
%\end{equation}
%where convergence holds in $H^{-s}$ for the measure and in $H^s$ for $\frac{\delta \phi_n}{\delta \mu}$ and $\frac{\delta \phi_n}{\delta \nu}.$
\label{Ishii'sLemma}
\end{lem}
We continue with the proof of the comparison principle. Using the sub-solution property, as in Lemma \ref{equivalencedefviscositysolutions}, for $U^{\gamma}$ at $(\bar{s}, \bar{y}, \bar{\mu})$
%see the comment before \eqref{equivalencedefviscositysolutions}
and we get
\begin{equation}
    -\frac{1}{\epsilon}(\bar{s} - \bar{t}) - a_0\tr{X} -\int_{\T^d} \Delta_x ( \bar{\mu} - \bar{\nu} )^{*}(x) d\bar{\mu}(x)  + \int_{\T^d} H \Bigl(x + \bar{y}, \frac{1}{\epsilon} ( \bar{\mu} - \bar{\nu} )^{*}, \bar{\mu} \Bigr) d\bar{\mu}(x) \leq -\gamma. 
\end{equation}
Arguing similarly for $V$ we find
\begin{equation}
    -\frac{1}{\epsilon}(\bar{s} - \bar{t}) - a_0\tr{Y} -\int_{\T^d} \Delta_x ( \bar{\mu} - \bar{\nu} )^{*}(x) d\bar{\nu}(x) + \int_{\T^d} H \Bigl(x + \bar{z}, \frac{1}{\epsilon} ( \bar{\mu} - \bar{\nu} )^{*} , \bar{\nu} \Bigr) d\bar{\nu}(x) \geq 0. 
\end{equation}
Subtracting the two inequalities and rearranging then leads to
\begin{equation}
\gamma \leq I_1 + I_2 + I_3
\label{tobecontradicted9octobre}
\end{equation}
where 
\begin{align*}
    I_1 &:=\int_{\T^d} H \Bigl( x + \bar{z}, \frac{1}{\epsilon}(\bar{\mu} - \bar{\nu})^{*}, \bar{\nu} \Bigr) d \bar{\nu}(x) - \int_{\T^d} H \Bigl( x + \bar{y}, \frac{1}{\epsilon}(\bar{\mu} - \bar{\nu})^{*}, \bar{\mu} \Bigr) d \bar{\mu}(x), \\
    I_2 &:=a_0 \tr(X-Y), \\
    I_3 &:= \int_{\T^d} \Delta_x ( \bar{\mu} - \bar{\nu} )^{*}(x) d( \bar{\mu} - \bar{\nu})(x).
\end{align*}
In view respectively of the matrix inequality \eqref{MatrixInequality9oct2023} in Lemma \ref{Ishii'sLemma} as well as the non-positivity of the Laplacian in $H^k$, we know that $I_2 \leq 0$ and $I_3 \leq 0$. It remains to estimate $I_1.$ We rewrite 
\begin{align*}
    &I_1 = I_{1,a} + I_{1,b}, \text{ with } \\
    &I_{1,a} \coloneqq \int_{\T^d} H \Bigl( x + \bar{z}, \frac{1}{\epsilon}(\bar{\mu} - \bar{\nu})^{*}, \bar{\mu} \Bigr) d (\bar{\nu} - \bar{\mu})(x), \\
    &I_{1,b} \coloneqq \int_{\T^d} \Big[ H\Bigl( x + \bar{z}, \frac{1}{\epsilon}(\bar{\mu} - \bar{\nu})^{*}, \bar{\nu} \Bigr) -H  \Bigl( x + \bar{y}, \frac{1}{\epsilon}(\bar{\mu} - \bar{\nu})^{*}, \bar{\mu} \Bigr) \Bigr] d \bar{\mu}(x).
\end{align*}
%\red{JJ: it is actually important not to assume $H$ is $C^s$-smooth in this part, so I have rewritten as follows}
On the one hand, using Assumption \ref{assump.d1} as well as \eqref{estimatefromregularity} (with $k \geq d/2+1)$), we see that 
\begin{align*}
    x \mapsto H\Big(x + \bar{z}, \frac{1}{\epsilon} (\bar{\mu} - \bar{\nu})^{*}(x), \bar{\mu} \Big)
\end{align*}
is, uniformly in $\epsilon$, H\"older continuous. Since $\| \bar{\mu} - \bar{\nu}\|_{-k} \rightarrow 0$ as $\epsilon \rightarrow 0$, we conclude easily that
%\red{maybe include details, but it is a simple mollification argument} that 
\begin{align*}
    \lim_{\epsilon \to 0^+} I_{1,a} = 0.
\end{align*}
On the other hand, using Assumption \eqref{assump.d1}
%\textcolor{blue}{assuming that $H$ satisfies an estimate of the form
%$$\bigl| H(x,p,\mu) - H(y,p,\nu) \bigr| \leq \omega_R \bigl(|x-y| + \bd_1(\mu, \nu) \bigr)$$
%for some modulus $\omega_R$,  for all $(x,y,\mu,\nu) \in (\T^d)^2 \times \bigl( \mathcal{P}(\T^d) \bigr)^2$ and all $p \in \R^d$ such that $\norm{p}_{\infty} \leq R,$ (which actually follows from Assumption \eqref{assump.d1})}
and recalling \eqref{estimateLinftyfromregularity} we find that 
$$|I_{1,b}| \leq C \bigl(| \bar{z} - \bar{y}| + \bd_1( \bar{\mu}, \bar{\nu}) \bigr),$$
for some $C>0$ independent from $\epsilon$. From \eqref{Penalizationgoesto0}, we have $\lim_{\epsilon \rightarrow 0} (\bar{z} - \bar{y} ) =0$ and $\lim_{\rightarrow 0} \norm{ \bar{\mu} - \bar{\nu} }_{-k} = 0$ which leads to
$$\lim_{\epsilon \rightarrow 0} I_{1,b} =0.$$
Coming back to \eqref{tobecontradicted9octobre} we find, letting $\epsilon \rightarrow 0$,
$\gamma \leq 0,$
which is the desired contradiction.
\end{proof}

It remains to prove Lemma \ref{Ishii'sLemma}, which we do similarly as in \cite{bayraktar2023}. We first prove the following elementary lemma.
%This approach relies on the following elementary observation, where we use the simple dot notation for the inner product of $H^{-k}.$

\begin{lem}
    Let $H$ be a Hilbert space with norm $\norm{\cdot}$ and the inner product of $\mu,\nu \in H$ denoted by $\mu \cdot \nu$. Take $\mu,\nu,\bar{\mu}, \bar{\nu}$ in $H$. Then
    $$ \norm{\mu-\nu}^2 \leq 2 \mu \cdot (\bar{\mu} - \bar{\nu}) - 2 \nu \cdot (\bar{\mu} - \bar{\nu}) +2 \norm{\mu- \bar{\mu} }^2 +2 \norm{\nu - \bar{\nu} }^2 
 - \norm{\bar{\mu} - \bar{\nu}}^2$$
 and inequality holds if and only if $\mu- \bar{\mu} = \nu - \bar{\nu}$.
 \label{HilbertianInequality}
\end{lem}

\begin{proof}
Developing the square $ \norm{\mu- \bar{\mu} + \bar{\nu} - \nu}^2$ and rearranging leads to
$$\norm{\mu-\nu}^2 + \norm{\bar{\mu} - \bar{\nu}}^2 = 2\mu\cdot(\bar{\mu} - \bar{\nu}) - 2\nu \cdot (\bar{\mu} - \bar{\nu}) + \norm{\mu-\bar{\mu} + \bar{\nu} - \nu}^2.$$
It remains to notice that
$$\norm{\mu - \bar{\mu} + \bar{\nu} - \nu}^2 \leq 2 \norm{\mu- \bar{\mu} }^2 + 2 \norm{\nu - \bar{\nu}}^2$$
with equality if and only if $\mu- \bar{\mu} = \nu - \bar{\nu}.$
\end{proof}

\begin{proof}[Proof of Lemma \ref{Ishii'sLemma}]
%\textcolor{blue}{Idea: up to doing sup/inf convolution in $z$ we can assume semi-concavity/semi-convexity for U and V in $z$ and then keep $\kappa=0$ so 4.10 is allright.}    
Up to introducing a further penalization of the form 
$$(s,t,y,z,\mu,\nu) \rightarrow  |s-\bar{s}|^4 + |t-\bar{t}|^4 + |y - \bar{y}|^4 + |z- \bar{z}|^4 + \norm{\mu - \bar{\mu}}_{-k}^4 + \norm{\nu - \bar{\nu}}_{-k}^4$$
and redefining $U$ and $V$ outside a neighborhood of $\bar{\theta}$ we can assume that the maximum is unique. 
%\textcolor{blue}{This is easily achieved using the smoothness of $(\mu,\nu) \mapsto \norm{\mu - \nu}^2_{-s}$ over $\mathcal{P}(\T^d) \times \mathcal{P}(\T^d)$, see Barles ... .} 

%\textcolor{blue}{I think we can also assume, for the same reason, ie by adding a penalization like $\norm{\theta - \bar{\theta}}^4$, that $ \bar{\theta}$ is a "very strict" maximum in the sense that any sequence $(t_n,s_n,z^1_n,z^2_n,\mu_n,\nu_n)$ such that
%$$\Phi^{\epsilon}(t_n,s_n,z^1_n,z^2_n,\mu_n,\nu_n) \geq \max \Phi^{\epsilon} - \delta_n$$
%for some sequence $\delta_n \rightarrow 0^+$ converges to $\bar{\theta}$ as $n \rightarrow +\infty$. And maybe it helps for the compactness issue below and avoids using a complicated variational principle.}

From Lemma \ref{HilbertianInequality} we deduce that $\bar{\theta} = (\bar{s} , \bar{t}, \bar{y}, \bar{z}, \bar{\mu}, \bar{\nu})$ is the unique maximum, over $[0,T]^2 \times (\T^d)^2 \times \big(\mathcal{P}(\T^d)\big)^2$, of 
\begin{equation}
    (s,t,y,z,\mu,\nu) \mapsto U^{\gamma}_1(s,y,\mu) - V_1(t,z,\nu) - \frac{1}{2\epsilon} |y- z|^2 - \frac{1}{2\epsilon} |s-t|^2 
\label{maxforU1V121Oct}
\end{equation}
where $U^{\gamma}_1$ and $V_1$ are defined by
\begin{align*}
    U^{\gamma}_1(s,y,\mu) &:= U^{\gamma}(s,y,\mu) - \frac{1}{\epsilon} \ip{ \mu , (\bar{\mu} - \bar{\nu} ) }_{-k} - \frac{1}{\epsilon} \norm{ \mu - \bar{\mu}}^2_{-k}, \\
    V_1(t,z,\nu) &:= V(t,z,\nu) + \frac{1}{\epsilon} \ip{ \nu , (\bar{\mu} - \bar{\nu}) }_{-k} + \frac{1}{\epsilon} \norm{ \nu - \bar{\nu} }^2_{-k}.
\end{align*}
Letting 
$$U^{\gamma}_0(s, y) := \sup_{\mu \in \mathcal{P}(\T^d)} U_1^{\gamma}(s,y,\mu),\quad V_0(t,z) := \inf_{\nu \in \mathcal{P}(\T^d)} V_1(t,z,\nu)   $$
we deduce that $(\bar{s} ,\bar{t} , \bar{y} , \bar{z})$ is the unique maximum, over $[0,T]^2 \times \bigl( \T^d \bigr)^2 $, of 
$$(s,t,y,z) \mapsto U^{\gamma}_0(s,y) - V_0(t,z) - \frac{1}{2\epsilon} |y - z|^2 - \frac{1}{2\epsilon} |s-t|^2.$$
Moreover, if $\mu', \nu' \in \mathcal{P}(\T^d)$ are such that
$$U_0^{\gamma}(\bar{s}, \bar{y}) = U_1^{\gamma}( \bar{s} , \bar{y}, \mu'), \quad V_0(\bar{t}, \bar{z}) = V_1( \bar{t} , \bar{z}, \nu'), $$
then, $(\bar{s} , \bar{t}, \bar{y}, \bar{z}, \mu', \nu')$ is a maximum of \eqref{maxforU1V121Oct} and therefore, $\mu' = \bar{\mu}$ and $\nu' = \bar{\nu}.$

Using the classical Ishii's Lemma, see \cite{Crandall1992}, we find, for every $\kappa >0$, matrices $X,Y \in S^d(\R)$ satisfying the matrix inequality \eqref{MatrixInequality9oct2023}, and such that 
\begin{equation}
\begin{array}{ll}
    \Bigl( \frac{1}{\epsilon}( \bar{s} - \bar{t} ) , \frac{1}{\epsilon} ( \bar{y} - \bar{z} ), X \Bigr) \in \bar{J}^{2,+} U_0^{\gamma} (\bar{s} , \bar{y}), \\
   \Bigl( \frac{1}{\epsilon}( \bar{s} - \bar{t} ) , \frac{1}{\epsilon} ( \bar{y} - \bar{z} ), Y \Bigr) \in \bar{J}^{2,-}V(\bar{t}, \bar{z}),  
\end{array}
\end{equation}
where semi-jets are understood in the classical parabolic sense for functions defined over $[0,T] \times \T^d$.
Now we take sequences $(s_n, y_n)_{n \geq 0}$, $(r_n, p_n, X_n)_{n \geq 0}$ converging respectively to $(\bar{s}, \bar{y})$ and $(\frac{1}{\epsilon}( \bar{s} - \bar{t}), \frac{1}{\epsilon}( \bar{y}- \bar{z}) , X)$ and such that
\begin{equation}
(r_n, p_n, X_n) \in J^{2,+} U_0^{\gamma}(s_n , y_n), \quad \forall n \geq 0,
\label{SuperJetU021oct}
\end{equation}
and
\begin{equation}
    \lim_{n \rightarrow +\infty} U_0^{\gamma}(s_n , y_n) = U_0^{\gamma}(\bar{s} , \bar{y}).
\end{equation}
We also take $\mu_n \in \mathcal{P}(\T^d)$ such that, for all $n \geq 0$,
\begin{equation}
U_0^{\gamma}(s_n, y_n) = U_1^{\gamma}(s_n,y_n,\mu_n).
\end{equation}
Since $U_1^{\gamma}$ is upper semi-continuous over $(0,T) \times \T^d \times \mathcal{P}(\T^d)$, every limit point $(\bar{s}, \bar{y}, \mu')$ of $(s_n,y_n,\mu_n)$ must verify
\begin{equation}
U_0^{\gamma}(\bar{s} , \bar{y}) = \lim_{n \rightarrow +\infty} U_1^{\gamma}(s_n , y_n, \mu_n) \leq U_1^{\gamma}( \bar{s} , \bar{y}, \mu')
\label{NoName21Oct2023}
\end{equation}
and therefore $\mu' = \bar{\mu}$ by strict maximality of $\bar{\mu}$ for $U^{\gamma}_0(\bar{t}, \bar{z}_1).$
By compactness of $\mathcal{P}(\T^d)$ we deduce that $(s_n, y_n, \mu_n)$ converges to $(\bar{s}, \bar{y}, \bar{\mu})$ as $n \rightarrow +\infty$ and, from \eqref{NoName21Oct2023}, we have
$$\lim_{n \rightarrow +\infty} U_1^{\gamma}(s_n,y_n,\mu_n) = U_1^{\gamma}(\bar{s} , \bar{y}, \bar{\mu} ).$$

Using \eqref{SuperJetU021oct} as well as the definitions of $U_1^{\gamma}$ and $U^{\gamma}_0$ find that there is some smooth function $\phi_n : [0,T] \times \T^d \rightarrow \R$ satisfying  
$$\bigl( \partial_t \phi_n, D \phi_n \bigr) (s_n,y_n) = (r_n,p_n), \quad D^2 \phi_n (s_n,y_n) \leq X_n,$$
such that
$$(s,y,\mu) \mapsto U^{\gamma}(s,y,\mu) - \frac{1}{\epsilon} \ip{\mu \cdot (\bar{\mu} - \bar{\nu})}_{-k} - \frac{1}{\epsilon} \norm{\mu - \bar{\mu}}^2_{-s} - \phi_n(s,y) $$
achieves its maximum at $(s_n , y_n, \mu_n)$ and therefore,
$$ \Bigl( r_n , p_n, X_n, \frac{1}{\epsilon}(\bar{\mu} - \bar{\nu})^{*} + \frac{2}{\epsilon}(\mu_n - \bar{\mu})^{*} \Bigr) \in J^{2,+} U^{\gamma}(s_n, y_n, \mu_n).$$
Since $(s_n, y_n, \mu_n)_{n \geq 0}$ converges to $(\bar{s} , \bar{y}, \bar{\mu})$ and $(r_n,p_n,X_n)_{n \geq 0}$ to $\Bigl(\frac{1}{\epsilon}( \bar{s} - \bar{t}), \frac{1}{\epsilon}( \bar{y}- \bar{z}) , X \Bigr) $  this leads to the first inclusion in \eqref{IshiisInclusion21oct}. We prove the second one similarly. This concludes the proof of the lemma.

\end{proof}

We now remove the assumption in Propostion \ref{prop.partialcomparison} that one of the solutions be Lipschitz continuous in the $H^{-k}$-metric, thereby completing the proof of the comparison principle in Theorem \ref{thm.comparison}.

\begin{proof}[Proof of Theorem \ref{thm.comparison}]
    Up to replacing $U^1$ by $U^1 - \delta (T-t)$ and $U^2$ by $U^2 + \delta (T-t)$ we may assume, without loss of generality, that we have
    \begin{align}
\begin{cases}
    \ds - \partial_t U^1 - (1 + a_0) \int_{\T^d} \tr(D_x D_m U^1\big)(t,m,x)  m(dx) - a_0 \int_{\T^d} \int_{\T^d} \tr\big(D_{mm} U^1\big)(t,m,x,y) m(dx) m(dy) \vspace{.1cm} \\ \ds
     \hspace{2cm} + \int_{\T^d} H\big(x, D_m U^1(t,m,x), m \big) m(dx) \leq - \delta, \quad (t,m) \in [0,T] \times \cP(\T^d), \vspace{.1cm} \\
    U^1(T,m) \leq  G(m) - \delta, \quad m \in \cP(\T^d)
    \end{cases}
\end{align}
in a viscosity sense for some $\delta > 0$, and likewise we have 
\begin{align}
\begin{cases}
    \ds - \partial_t U^2 - (1 + a_0) \int_{\T^d} \tr(D_x D_m U^2\big)(t,m,x)  m(dx) - a_0 \int_{\T^d} \int_{\T^d} \tr\big(D_{mm} U^2\big)(t,m,x,y) m(dx) m(dy) \vspace{.1cm} \\ \ds
     \hspace{2cm} + \int_{\T^d} H\big(x, D_m U^2(t,m,x), m \big) m(dx) \geq  \delta, \quad (t,m) \in [0,T] \times \cP(\T^d), \vspace{.1cm} \\
    U^2(T,m) \geq  G(m) + \delta, \quad m \in \cP(\T^d).
    \end{cases}
\end{align}
    We now apply Lemma \ref{lem.mollification} to produce $(H^{n}, G^{n})_{n \in \N}$ satisfying conditions (1)-(3) in the statement of Lemma \ref{lem.mollification}. Next, let $U^{n}$ denote a solution to \eqref{hjb} but with $G$ and $H$ replaced by $G^{n}$ and $H^{n}$. Thanks to Proposition \ref{prop.existsmooth}, we may assume that $U^{n}$ satisfies \eqref{regintro} for each $k \in \N$, and also that $U^{n}$ satisfy \eqref{d1regintro} uniformly in $n$. Now let $\ov{U}^n(t,z,m) = U(t,m^z)$. First, we note that the functions $\ov{U}^n$ are $\bd_1$-Lipschitz in $m$, uniformly in $n$. As a consequence, we can find a constant $C_0$ with the following property: for any $\Phi \in \cpart$ and any $(t_0,z_0,m_0)$ such that $\ov{U}^n - \Phi$ is maximized at $(t_0,z_0,m_0)$, we have $|D_m \Phi(t_0,z_0,m_0,x)| \leq C_0$ for each $x \in \text{supp}(m_0)$ (see e.g. Proposition 5.1 in \cite{Soner2023} for a simple proof). Now let $c_n$ be a sequence with $c_n \downarrow 0$ such that 
    \begin{align*}
        \sup_{(x,p,m) : |p| \leq C_0} |H(x,p,m) - H^n(x,p,m)| \leq c_n, \text{ and } \sup_m |G(m) - G^n(m)| \leq c_n.
    \end{align*}
    Then we have
    \begin{align*} 
        \ds &- \partial_t \Phi(t_0,z_0,m_0) - \int_{\T^d} \tr(D_x D_m \Phi \big)(t_0,z_0,m_0,x)  m_0(dx) - a_0 \Delta_z \Phi (t_0,z_0,m_0) \vspace{.1cm} \nonumber \\
    \ds &\hspace{2cm} + \int_{\T^d} \ov{H}\big(x, D_m \Phi(t_0,z_0,m_0,x), m_0,z_0 \big) m_0(dx) \\
    \ds &\leq - \partial_t \Phi(t_0,z_0,m_0) - \int_{\T^d} \tr(D_x D_m \Phi \big)(t_0,z_0,m_0,x)  m_0(dx) - a_0 \Delta_z \Phi (t_0,z_0,m_0) \vspace{.1cm} \nonumber \\
    \ds &\hspace{2cm} + \int_{\T^d} \ov{H}^n\big(x, D_m \Phi(t_0,z_0,m_0,x), m_0,z_0 \big) m_0(dx) + c_n \leq c_n,
    \end{align*}
    and also $\ov{U}^n(T,z,m) \leq G^n(t,m^z) + c_n$.

In particular, choosing $n$ large enough that $c_{n} < \delta$, we can use Proposition \ref{prop.partialcomparison} to conclude that, for large enough $n$, $U^1 \leq U^n \le U^2$, and so in particular $U^1 \leq U^2$, as required.

%\textcolor{red}{I think we should give more details here. Especially to justify why $c_n \rightarrow 0$ thanks to the uniform $d_1$-Lipschitz regularity. JJ: let me know what you think of the above.  SD: wonderful !}
\end{proof}

\begin{proof}[Proof of Theorem \ref{thm.regularity}]
    Proposition \ref{prop.existsmooth} shows that there exists a solution to \eqref{hjb} satisfying the desired regularity estimates, while Theorem \ref{thm.comparison} shows that this is the unique solution.
\end{proof}

\section{Applications} \label{sec.applications}

\subsection{Mean-field control with common noise}
\label{subsec.mfc}
The goal of this subsection is to show that
\begin{enumerate}
    \item Theorem \ref{thm.comparison} allows us to identify the value function of a mean field control problem with common noise as the unique viscosity solution (in the sense of Definition \ref{def.viscosity}) of the corresponding HJB equation, and 
    \item Theorem \ref{thm.convergence} yields a new analytical proof that the value functions of certain finite-dimensional control problems converge to their mean field counterparts.
\end{enumerate}
In order to do so, we will need to use the dynamic programming principle from \cite{djete2022aop}, and so we take care to formulate the mean field problem in a way which is compatible with the main results in \cite{djete2022aop}. The data in this setting will consist of 
\begin{align*}
     (b,L) : \T^d \times A \times \cP(\T^d) \to \R^d \times \R, \quad G : \cP(\T^d) \to \R,
\end{align*}
where $A$ is, for simplicity, a compact metric space. We set define $H : \T^d \times \R^d \times \cP(\T^d) \to \R$ by
\begin{align} \label{hamdefcontrol}
    H(x,p,m) = \sup_{a \in A} \bigl \{ - L(x,a,m) - b(x,a,m) \cdot p \bigr \}.
\end{align}

We assume the following on the data:

\begin{assumption} \label{assump.control}
The functions $L$ and $b$ are continuous, and Lipschitz in $(x,m)$, uniformly in $a$. The function $G$ is Lipschitz. 
\end{assumption}

\begin{rmk}
    Under Assumption \ref{assump.control}, it is easy to check that $H$ and $G$ satisfy Assumption \ref{assump.d1}. It should be possible to relax the assumption that $L$ is uniformly Lipschitz in $(x,m)$, but this Assumption is needed to apply the results in \cite{djete2022aop}.
\end{rmk}

To define the mean field value function, we work on the canonical space 
\begin{align*}
    \Omega^0 \coloneqq \T^d \times \cC \times \cC, \text{ with } \cC = C([0,T] ; \R^d). 
\end{align*}
For each $\mu \in \cP(\T^d)$, let $\bP^{\mu}$ denote the probability measure
\begin{align*}
    \bP^{\mu} = \mu \otimes \bP^0 \otimes \bP^0, 
\end{align*}
where $\bP^0$ is the Wiener measure on $C$. Write $(x, \omega, \omega^0)$ for the general element of $\Omega^0$, and denote by $(\xi, W, W^0)$ the canonical random variable and processes on $\Omega^0$, i.e. 
\begin{align*}
\xi(x,\omega, \omega^0) = x, \quad W_t(x,\omega, \omega^0) = \omega(t), \quad  W^0_t(x,\omega, \omega^0) = \omega^0(t).
\end{align*}
Let $\bbF = (\sF_t)_{0 \leq t \leq T}$ denote the right-continuous augmentation of the filtration generated by $\xi$, $W$, and $W^0$, and let $\bbF^0 = (\sF_t^0)_{0 \leq t \leq T}$ denote the right-continuous augmentation of the filtration generated by $W^0$. With $(t_0,m_0) \in [0,T] \times \cP(\T^d)$ fixed, we denote by $\cA_{t_0,m_0}$ the set of $A$-valued processes $\alpha = (\alpha_t)_{t_0 \leq t \leq T}$, progressively measurable with respect to $\bbF$. Then we define 
\begin{align} \label{udefcontrol}
    U(t_0,m_0) = \inf_{\alpha \in \cA_{t_0,m_0}} \E^{\bP^{m_0}}\bigg[ \int_{t_0}^T L\big(X_t^{t_0,\alpha}, \alpha_t, \sL^{0,m_0}(X_t^{t_0,\alpha}) \big) dt + G\big(\sL^{0,m_0}(X_{T}^{t_0,\alpha})\big)\bigg],
\end{align}
where for each $\alpha \in \cA_{t_0, m_0}$, $X^{t_0,\alpha}$ is defined by 
\begin{align*}
    dX^{t_0,\alpha}_t = b \bigl(X_t^{t_0,\alpha}, \alpha_t, \sL^{0,m_0}(X_t^{t_0,\alpha}) \bigr) dt + \sqrt{2} dW_t + \sqrt{2a_0} dW_t^0, \quad t_0 \leq t \leq T, \quad X^{t_0,\alpha}_{t_0} = \xi,
\end{align*}
and where for a random variable $X$, $\sL^{0,m_0}$ denotes the law of $X$ conditional on $\sF^0_T$ and with respect to $\bP^{m_0}$.

The definition of the finite-dimensional value functions $V^N : [0,T] \times (\T^d)^N \rightarrow \R$ is much more standard. In particular, for each $N$ we fix a filtered probability space $(\Omega^N, \sF^N, \bbF^N = (\sF^N_t)_{0 \leq t \leq T}, \bP^N)$ satisfying the usual conditions and hosting independent $d$-dimensional Brownian motions $W^0,W^1,...,W^N$. We define 
\begin{align} \label{def.ndimvalues}
    V^N(t_0,\bx_0) = \inf_{\bm \alpha = (\alpha^1,...,\alpha^N)} \E \bigg[\int_{t_0}^T \frac{1}{N} \sum_{i = 1}^N L(X_t^i,\alpha_t^i, m_{\bX_t}^N) dt + G(m_{\bX_T}^N) \bigg]
\end{align}
where we recall the $m_{\bX_t}^N$ stands for the empirical measure $\frac{1}{N} \sum_{i=1}^N \delta_{X_t^{i}},$ and
where the infimum is taken over all $\bbF^N$-progressive $A^N$-valued processes $\bm \alpha = (\alpha^1,...,\alpha^N)$, and the state process $\bX = (X^1,...,X^N)$ is determined from $(t_0,\bx_0)$ and the control $\bm \alpha$ via the dynamics
\begin{align*}
    dX_t^i = b(X_t^i, \alpha_t^i, m_{\bX_t}^N) dt + \sqrt{2} dW_t^i + \sqrt{2a_0} dW_t^{0}, \quad t_0 \leq t \leq T, \quad X_{t_0}^i = x_0^i.
\end{align*}
Under Assumption \ref{assump.control}, $V^N$ is the unique viscosity solution of \eqref{hjbn}, with $H$ given by \eqref{hamdefcontrol}.

\begin{prop} \label{prop.control} Under Assumption \ref{assump.control}, the function $U$ defined by \eqref{udefcontrol} is the unique viscosity solution of \eqref{hjb}, with $H$ defined by \eqref{hamdefcontrol}.
\end{prop}

\begin{proof}
    By Corollary 3.11 of \cite{djete2022aop}, $U$ satisfies a dynamic programming principle: for each $(t_0,m_0) \in [0,T) \times \cP(\T^d)$, $0 < h < T- t_0$, we have 
    \begin{align} \label{dppu}
        U(t_0,m_0) = \inf_{\alpha \in \cA_{t_0,m_0}} \E\bigg[ \int_{t_0}^{t_0 + h} L\big(X^{t_0,\alpha}_t, \alpha_t, \sL^{0,m_0}(X^{t_0,\alpha}_t) \big) dt + U\big(t_0 + h, \sL^{0,m_0}(X^{t_0,\alpha}_{t_0 + h})\big)\bigg].
    \end{align}

    Next, for each $(t_0,m_0) \in [0,T], z_0 \in \T^d$ $\alpha \in \cA_{t_0,m_0}$, we define the processes $Z^{t_0,z_0}$ and $\ov{X}^{t_0,z_0,m_0}$ by 
    \begin{align*}
        Z^{t_0,z_0}_t &= z_0 + \sqrt{2a_0} (W_t^0 - W_{t_0}), \quad t_0 \leq t \leq T, \\
        \ov{X}^{t_0,z_0}_t &= X^{t_0,\alpha}_t - Z^{t_0,z_0}_t, \quad t_0 \leq t \leq T.
    \end{align*}
    We notice that because $Z^{t_0,z_0}$ is $\sF^0_T$-measurable, we have 
    \begin{align*}
        \sL^{0,m_0}(X_t^{t_0,\alpha}) = (\text{Id} + Z_t^{t_0,z_0} )_{\#} \sL^{0,m_0}(\ov{X}_t^{t_0,z_0,\alpha}).
    \end{align*}
    Using this observation, one easily checks that the function $\ov{U} : [0,T] \times \T^d \times \R^d \to \R$ defined by 
    \begin{align} \label{uovu}
        \ov{U}(t,z,m) = U(t,m^z)
    \end{align}
    in fact satisfies
    \begin{align*}
        \ov{U}(t_0,z_0,m_0) &= \inf_{\alpha \in \cA_{t_0,m_0}} \E^{\bP^{m_0}}\bigg[ \int_{t_0}^T L\Big(\ov{X}_t^{t_0,\alpha} + Z_t^{t_0,m_0}, \alpha_t, (\text{Id} + Z_t^{t_0,m_0})_{\#} \sL^{0,m_0}(X_t^{t_0,\alpha}) \Big) dt  \\
        &\hspace{6cm}+ G\Big((\text{Id} + Z_t^{t_0,m_0})_{\#} \sL^{0,m_0}(\ov{X}_T^{t_0,z_0,\alpha})\Big)\bigg]. 
    \end{align*}
    Moreover, from \eqref{dppu} and the identity \eqref{uovu}, we can deduce the following dynamic programming principle for $\ov{U}$:
    \begin{align} \label{dppovu}
        \ov{U}(t_0,z_0,m_0) &= \inf_{\alpha \in \cA_{t_0,m_0}} \E^{\bP^{m_0}}\bigg[ \int_{t_0}^{t_0 + h} L\Big(\ov{X}_t^{t_0,\alpha} + Z_t^{t_0,m_0}, \alpha_t, (\text{Id} + Z_t^{t_0,m_0})_{\#} \sL^{0,m_0}(\ov{X}_t^{t_0,z_0,\alpha}) \Big) dt \nonumber  \\
        &\hspace{6cm}+ \ov{U}\Big(t_0 + h, Z_t^{t_0,m_0}, \sL^{0,m_0}(\ov{X}_T^{t_0,z_0,\alpha})\Big)\bigg]. 
    \end{align}
    From here, it is straightforward to check that $\ov{U}$ is a viscosity solution of \eqref{hjbz}, since \eqref{hjbz} is nothing but the dynamic programming equation corresponding to the dynamic programming principle \eqref{dppovu}. Thus $U$ is in fact the unique viscosity solution to \eqref{hjb}, as claimed.
\end{proof}

\begin{cor}
    Under Assumption \ref{assump.control}, the value functions $V^N$ defined by \eqref{def.ndimvalues} converge to the function $U$ defined by \eqref{udefcontrol}, in the sense that 
    \begin{align*}
        \sup_{(t,\bx) \in [0,T] \times (\T^d)^N} |V^N(t,\bx) - U(t,m_{\bx}^N)| \xrightarrow{N \to \infty} 0.
    \end{align*}
\end{cor}

\subsection{A zero-sum game on the Wasserstein space} \label{subsec.zerosum}

Here we consider a zero-sum game, which is formulated as follows. First, the game will be set on a fixed filtered probability space $(\Omega, \sF, \bbF, \bP)$ satisfying the usual conditions and hosting a $d$-dimensional Brownian motion $W = (W_t)_{0 \leq t \leq T}$. We furthermore assume that $\sF_0$ is atomless, and $\bbF$ is the (augmentation of) the filtration generated by $\sF_0$ and $W$. For simplicity, we fix compact metric spaces $A^1$ and $A^2$. The data of the game consists of functions 
\begin{align*}
    (b, L)(x,a^,a^2,m) : \T^d \times A_1 \times A_2 \times \cP(\T^d) \to \R^d \times \R, \quad G : \cP(\T^d) \to \R.
\end{align*}
We make the following assumptions on the data. 
\begin{assumption} \label{assump.games}
   The functions $b$ and $L$ are continuous, and Lipschitz in $(x,m)$ uniformly in $a^1,a^2$. The function $G$ is Lipschitz.
\end{assumption}
For each $t_0 \in [0,T]$, $i = 1,2$ we denote by $\cA^i_{t_0}$ the set of all progressive processes $\alpha^i$ defined on $[t_0,T]$ and taking values in $A^i$. Informally, the game is played as follows. We fix $(t_0,m_0) \in [0,T] \times \cP(\T^d)$, and then choose a $\sF_0$-measurable random variable $\xi$ with $\xi \sim m_0$. The dynamics of the state process $X$ is given by
\begin{align} \label{xdefgame}
    dX_t = b\big(X_t, \alpha_t^1, \alpha^2_t, \sL(X_t) \big) dt + \sqrt{2}dW_t, \quad t_0 \leq t \leq T, \quad X_{t_0} = \xi,
\end{align}
and the payoff is determined by a functional 
\begin{align} \label{jdef}
J(t_0,\xi, \alpha^1, \alpha^2) = \E\bigg[ \int_{t_0}^T L\big(X_t, \alpha_t^1,\alpha_t^2, \sL(X_t) \big) dt + G(X_T)\bigg], \quad \text{where $X$ is given by \eqref{xdefgame}.}
\end{align}
Player I chooses $\alpha^1 \in \cA^1_{t_0}$ in order to maximize $J$, while Player II chooses $\alpha^2 \in \cA^2_{t_0}$ in order to minimize $J$. Of course, as is well known, there is some subtlety in formalizing this zero-sum game, which can be overcome by introducing the notion of strategies. We will denote by $\cA_{t_0}^{1,\text{s}}$ the set of strategies for player I, i.e. the set of all maps 
$$ \alpha^{1,s}[\alpha^2] : \mathcal{A}_{t_0}^2 \rightarrow \mathcal{A}_{t_0}^1 $$
which are non-anticipative, in the sense that 
\begin{align*}
    \bP\Big[\alpha^2_r = \ov{\alpha^{2}}_r \text{ for a.e. } r \in [t_0,t] \Big] = 1  \implies \bP \Big[\alpha^{1,\text{s}} \big[\alpha^2 \big]_r = \alpha^{\textcolor{red}{1,}\text{s}} \big[\ov{\alpha^{2}} \big]_r \text{ for a.e. } r \in [t_0,t] \Big] = 1.
\end{align*}

The set $\cA_{t_0}^{2,\text{s}}$ of strategies for player 2 is defined in a similar way. The lower value function $U^{-} : [t_0,T] \times \cP(\T^d) \to \R$ is defined by 
\begin{align} \label{uminusdef}
U^-(t_0,m_0) = \inf_{\alpha^\text{2,\text{s}} \in \cA^{2,\text{s}}_{t_0}} \sup_{\alpha^1 \in \cA^1_{t_0}} J\big(t,\xi, \alpha^1, \alpha^{2,\text{s}}\big[ \alpha^1 \big]\big), \quad \text{ where $\xi$ is $\sF_0$-measurable and } \xi \sim m_0
\end{align}
and the upper value functions is defined similarly by 
\begin{align} \label{uplusdef}
U^+(t_0,m_0) = \sup_{\alpha^{1,\text{s}} \in \cA^{1,\text{s}}_{t_0}} \inf_{\alpha^2 \in \cA^2_{t_0}} J\big(t,\xi, \alpha^{1,\text{s}}\big[\alpha^2 \big], \alpha^2 \big), \quad \text{ where $\xi$ is $\sF_0$-measurable and } \xi \sim m_0.
\end{align}
The fact that $U^+$ and $U^-$ are well-defined (i.e. the formulas in \eqref{uminusdef} and \eqref{uplusdef} are independent of the representative $\xi$) is established in \cite{cossophamjmpa}. It is also easy to check under Assumption \ref{assump.games} that $U^+$ and $U^-$ are both Lipschitz in $m$ (with respect to $\bd_1$) and 1/2-H\"older continuous in time, thanks to the uniform Lipschitz regularity of $L$ and $G$ in $(x,m)$. It is also shown in \cite[Corollary 4.1]{cossophamjmpa} that $U^+$ and $U^-$ satisfy the natural dynamic programming principles: 
\begin{align*}
    U^-(t_0,m_0) &= \inf_{\alpha^\text{2,\text{s}} \in \cA^{2,\text{s}}_{t_0}} \sup_{\alpha^1 \in \cA^1_{t_0}}  \bigg\{ \E\bigg[ \int_{t_0}^{t_0 + h} L\big(X_t, \alpha_t^1,\alpha^{2,\text{s}}\big[\alpha^1\big]_t, \sL(X_t) \big) dt\bigg] + U^-\big(t_0 + h, \sL(X_{t_0 + h}) \big) \bigg\}, \vspace{.1cm} \\
    &\text{ where $X$ is determined by \eqref{xdefgame} with $\beta = \beta^{\text{s}}\big[\alpha\big]$ and $\xi \sim m_0$ is $\sF_0$-measurable,}
\end{align*}
and similarly
\begin{align*}
    U^+(t_0,m_0) &= \sup_{\alpha^{1,\text{s}} \in \cA^{1,\text{s}}_{t_0}} \inf_{\alpha^2 \in \cA^2_{t_0}} \bigg\{ \E\bigg[ \int_{t_0}^{t_0 + h} L\big(X_t, \alpha^{1,\text{s}}\big[\alpha^2\big]_t,\alpha^2_t, \sL(X_t) \big) dt\bigg] + U^-\big(t_0 + h, \sL(X_{t_0 + h}) \big) \bigg\}, \vspace{.1cm} \\
    &\text{ where $X$ is determined by \eqref{xdefgame} with $\alpha = \alpha^{\text{s}}\big[\beta\big]$ and $\xi \sim m_0$ is $\sF_0$-measurable.}
\end{align*}
Using these dynamic programming principles, it is proved in \cite[Section 5]{cossophamjmpa}) that $U^+$ and $U^-$ are viscosity solutions of 
\begin{align} \label{hjbiminus}
\begin{cases}
    \ds - \partial_t U^- -\int_{\T^d} \tr(D_x D_m U^-\big)(t,m,x)  m(dx) \vspace{.1cm}  + \int_{\T^d} H^-\big(x, D_m U(t,m,x), m \big) m(dx) = 0,  \\
    \hspace{8cm} (t,m) \in [0,T] \times \cP(\T^d), \vspace{.1cm} \\
    U(T,m) = G(m), \quad m \in \cP(\T^d), 
    \end{cases}
\end{align}
and
\begin{align} \label{hjbiplus}
\begin{cases}
    \ds - \partial_t U^+ -\int_{\T^d} \tr(D_x D_m U^-\big)(t,m,x)  m(dx) 
 + \int_{\T^d} H^+\big(x, D_m U(t,m,x), m \big) m(dx) = 0, \vspace{.1cm} \\
     \hspace{8cm} (t,m) \in [0,T] \times \cP(\T^d), \vspace{.1cm} \\
    U(T,m) = G(m), \quad m \in \cP(\T^d), 
    \end{cases}
\end{align}
where 
\begin{align*}
    H^-(x,p,m) &= \min_{a \in A} \max_{b \in B} \Big\{- L(x,a,b,m) - h(x,a,b,m) \cdot p \Big\}, \vspace{.1cm} \\
    H^+(x,p,m) &= \max_{b \in A} \min_{a \in A} \Big\{- L(x,a,b,m) - h(x,a,b,m) \cdot p \Big\},
\end{align*}
at least in the sense of Remark \ref{rmk.nocommon}.

Next, we define a corresponding $N$-particle zero sum game. As in the previous subsection, for each $N$ we fix a filtered probability space $(\Omega^N, \sF^N, \bbF^N = (\sF^N_t)_{0 \leq t \leq T}, \bP^N)$ satisfying the usual conditions and hosting independent $d$-dimensional Brownian motions $W^0,W^1,...,W^N$. For $N \in \N$ and $t_0 \in [0,T]$, we denote by $\cA^{1,N}_{t_0}$ the set of $\bbF^N$-progressive processes $\bm \alpha^1 = (\alpha^{1,1},...,\alpha^{1,N})$ defined on $[t_0,T]$ and taking values in $(A^1)^N$, and we denote by $\cA^{2,N}$ the set of progressive processes $\bm \alpha^2 = (\alpha^{2,1},...,\alpha^{2,N})$ defined on $[t_0,T]$ and taking values in $(A^2)^N$. We will denote by $\cA_{t_0}^{1,\text{s},N}$ the set of all maps 
\begin{align*}
    \bm \alpha^{1,\text{s}}\big[\bm \alpha^2 \big] : \cA_{t_0}^{1,N} \to \cA_{t_0}^{2,N}
\end{align*}
which are non-anticipative, in the sense that 
\begin{align*}
    \bP\Big[\bm{\alpha}^{2}_t = \ov{\bm\alpha^2}_r \text{ for a.e. } r \in [t_0,t] \Big] = 1  \implies \bP \Big[\bm{\alpha}^{1,\text{s}} \big[\bm{\alpha}^{2}_t \big]_r = \bm\alpha^{1,\text{s}} \big[\ov{\bm\alpha^2}_r\big]_r \text{ for a.e. } r \in [t_0,t] \Big] = 1.
\end{align*}

The set $\cA_{t_0}^{2,\text{s},N}$ is defined in a similar way. The state process for our $N$-player game will be a $(\T^d)^N$-valued process $\bX = (X^1,...,X^N)$ determined by the dynamics 
\begin{align} \label{xdefngame}
    dX_t^i = b(X_t^i, \alpha_t^{1,i}, \alpha_t^{2,i}, m_{\bX_t}^N) dt + \sqrt{2} dW_t^i, \quad t_0 \leq t \leq T, \quad X_{t_0}^i = x^i.
\end{align}
In particular, we define the upper and lower $N$-particle value functions
\begin{align*}
    V^{N,+}, \,\, V^{N,-} : [t_0,T] \times (\T^d)^N \to \R, 
\end{align*}
by the formulas
\begin{align*}
    V^{N,-}(t,\bx) &= \inf_{\bm \alpha^{2,\text{s}} \in \cA^{2,s,N}_{t_0}} \sup_{\bm \alpha^1 \in \cA^{1,N}_{t_0}} J\big(t_0, \bx_0, \bm \alpha^1, \bm \alpha^{2,\text{s}}\big[ \bm \alpha^1 \big] \big), \\
    V^{N,+}(t,\bx) &= \sup_{\bm \alpha^{1,\text{s}} \in \cA^{1,s,N}_{t_0}} \inf_{\bm \alpha^2 \in \cA^{2,N}_{t_0}} J\big(t_0, \bx_0, \bm \alpha^{1,\text{s}}\big[ \bm \alpha^2 \big] , \bm \alpha^2 \big),
\end{align*}
where for any $(t_0, \bx_0) \in [0,T] \times (\T^d)^N$, $\bm \alpha \in \cA_{t_0}^N$, $\bm \beta \in \cB_{t_0}^N$, we have set
\begin{align*}
    J(t_0,\bx_0, \bm \alpha, \bm \beta) &= \E\bigg[\int_{t_0}^T \frac{1}{N}\sum_{i = 1}^N L\Big(X_t^i, \alpha_t^i, \beta^i_t, m_{\bX_t}^N \Big) dt + G(m_{\bX_T}^N) \bigg], \\
    &\qquad \text{where $X$ is given by \eqref{xdefngame}}.
\end{align*}
It is well-known, thanks to a string of results which can be traced back to \cite{flemingsoug1989}, that $V^{N,-}$ and $V^{N,+}$ are also characterized as the unique viscosity solutions of the PDEs
\begin{align} \label{hjbinminus} 
\begin{cases}
    \ds - \partial_t V^{N,-} - \sum_{i = 1}^N \Delta_{x^i} V^{N,-}  + \frac{1}{N} \sum_{i = 1}^N H(x^i, ND_{x^i} V^{N,-}, m_{\bx}^N) = 0, \vspace{.1cm} \\
    \hspace{8cm} (t,\bx) \in [0,T] \times (\T^d)^N, \vspace{.1cm} \\
    V^{N,-}(T,\bx) = G(m_{\bx}^N), \quad \bx \in (\T^d)^N.
    \end{cases}
\end{align}
and 
\begin{align} \label{hjbinplus} 
\begin{cases}
    \ds - \partial_t V^{N,+} - \sum_{i = 1}^N \Delta_{x^i} V^{N,+} + \frac{1}{N} \sum_{i = 1}^N H(x^i, ND_{x^i} V^{N,+}, m_{\bx}^N) = 0, \vspace{.1cm} \\
    \hspace{8cm} (t,\bx) \in [0,T] \times (\T^d)^N, \vspace{.1cm} \\
    V^{N,+}(T,\bx) = G(m_{\bx}^N), \quad \bx \in (\T^d)^N.
    \end{cases}
\end{align}
We thus arrive at the following consequences of our main results.
\begin{cor} \label{cor.convupperandlower}
    Suppose that Assumption \ref{assump.games} holds. Then we have 
    \begin{align*}
        &\sup_{(t,\bx) \in [0,T] \times (\T^d)^N} |V^{N,-}(t,\bx) - U^-(t,m_{\bx}^N)| \xrightarrow{N \to \infty} 0, \text{ and } \\
        &\sup_{(t,\bx) \in [0,T] \times (\T^d)^N} |V^{N,+}(t,\bx) - U^+(t,m_{\bx}^N)| \xrightarrow{N \to \infty} 0.
    \end{align*}
\end{cor}
\begin{proof}
    This is a consequence of our convergence result, Theorem \ref{thm.comparison} (and Remark \ref{rmk.nocommon}), together with Theorem 5.1 in \cite{cossophamjmpa}.
\end{proof}
\begin{cor} \label{cor.existenceofvalue}
    Suppose that Assumption \ref{assump.games} holds, and that $H^+ = H^-$ (i.e. the Isaacs condition holds). Then $U^+ = U^-$. 
\end{cor}
\begin{proof}
    This is a consequence of our main comparison result, Theorem \ref{thm.comparison}, (and Remark \ref{rmk.nocommon}), together with Theorem 5.1 in \cite{cossophamjmpa}.
\end{proof}

\begin{rmk} \label{rmk.cossophamlifting}
    In Remark 5.5 in of \cite{cossophamjmpa}, it is claimed that a uniqueness result can be obtained for the equations \eqref{hjbiplus} and \eqref{hjbiminus} by the ``lifting approach", but as mentioned in the introduction there are serious obstacles to using the lifting approach in the presence of idiosyncratic noise. Thus while we are relying completely on Theorem 5.1 of \cite{cossophamjmpa} to verify that the upper and lower value functions solve the upper and lower Hamilton-Jacobi-Isaacs equations \eqref{hjbiminus} and \eqref{hjbiplus}, our Corollary \ref{cor.existenceofvalue} establishes rigorously for the first time the existence of a value under the Isaacs condition.
\end{rmk}

\bibliographystyle{alpha}
\bibliography{HJBrefs}

\newcommand{\etalchar}[1]{$^{#1}$}
\begin{thebibliography}{CGK{\etalchar{+}}21}

\bibitem[BC18]{Briani2018}
Ariela Briani and Pierre Cardaliaguet.
\newblock {Stable solutions in potential mean field game systems}.
\newblock {\em Nonlinear Differential Equations and Applications}, 25(1):1--26,
  2018.

\bibitem[Ber]{bertucci2023}
Charles Bertucci.
\newblock Stochastic optimal transport and hamilton-jacobi-bellman equations on
  the set of probability measures.
\newblock Preprint,arXiv:2306.04283 [math.AP].

\bibitem[BEZ]{bayraktar2023}
Erhan Bayraktar, Ibrahim Ekren, and Xin Zhang.
\newblock Comparison of viscosity solutions for a class of second order pdes on
  the wasserstein space.
\newblock Preprint,arXiv:2309.05040 [math.AP].

\bibitem[BEZ23]{bayraktaretalvariational}
Erhan Bayraktar, Ibrahim Ekren, and Xin Zhang.
\newblock A smooth variational principle on wasserstein space.
\newblock {\em Proceedings of the American Mathematical Society},
  151:4089--4098, 2023.

\bibitem[CD18]{CarmonaDelarue_book_I}
Ren\'e Carmona and Fran{\c c}ois Delarue.
\newblock {\em Probabilistic Theory of Mean Field Games with Applications {I} :
  Mean Field {FBSDE}s, Control, and Games}.
\newblock Springer, 2018.

\bibitem[CDJS]{CDJS_23}
Pierre Cardaliaguet, Samuel Daudin, Joe Jackson, and Panagiotis Souganidis.
\newblock An algebraic convergence rate for the optimal control of
  mckean-vlasov dynamics.
\newblock Preprint, arXiv:2203.14554 [math.OC].

\bibitem[CGK{\etalchar{+}}21]{Cosso2021}
Andrea Cosso, Fausto Gozzi, Idris Kharroubi, Huy{\^{e}}n Pham, and Mauro
  Rosestolato.
\newblock {Master Bellman equation in the Wasserstein space: Uniqueness of
  viscosity solutions}.
\newblock {\em arXiv:2107.10535}, 2021.

\bibitem[CIL92]{Crandall1992}
Michael~G. Crandall, Hitoshi Ishii, and Pierre-Louis Lions.
\newblock {User's guide to viscosity solutions of second order partial
  differential equations}.
\newblock {\em Bulletin of the American Mathematical Society}, 27(1):1--67,
  1992.

\bibitem[CKT23]{confortihj}
G.~Conforti, R.C. Kraaij, and D.~Tonon.
\newblock Hamilton–jacobi equations for controlled gradient flows: The
  comparison principle.
\newblock {\em Journal of Functional Analysis}, 284(9):109853, 2023.

\bibitem[CM23]{cossomartini}
Andrea Cosso and Mattia Martini.
\newblock {On smooth approximations in the Wasserstein space}.
\newblock {\em Electronic Communications in Probability}, 28(none):1 -- 11,
  2023.

\bibitem[CP19]{cossophamjmpa}
Andrea Cosso and Huyên Pham.
\newblock Zero-sum stochastic differential games of generalized mckean–vlasov
  type.
\newblock {\em Journal de Mathématiques Pures et Appliquées}, 129:180--212,
  2019.

\bibitem[DDJ]{DDJ_23}
Samuel Daudin, Fran\c{c}ois Delarue, and Joe Jackson.
\newblock On the optimal rate for the convergence problem in mean field
  control.
\newblock Preprint, arXiv:2305.08423 [math.OC].

\bibitem[DPT22a]{djete2022mor}
Mao~Fabrice Djete, Dylan Possomai, and Xiaolu Tan.
\newblock Mckean–vlasov optimal control: Limit theory and equivalence between
  different formulations.
\newblock {\em Mathematics of Operations Research}, 47(4):2891--2930, 2022.

\bibitem[DPT22b]{djete2022aop}
Mao~Fabrice Djete, Dylan Possomai, and Xiaolu Tan.
\newblock {McKean–Vlasov optimal control: The dynamic programming principle}.
\newblock {\em The Annals of Probability}, 50(2):791 -- 833, 2022.

\bibitem[DS23]{daudinseeger}
Samuel Daudin and Benjamin Seeger.
\newblock A comparison principle for semilinear hamilton-jacobi-bellman
  equations in the wasserstein space.
\newblock Preprint,arXiv:2308.15174 [math.AP], 2023.

\bibitem[FG15]{fournier2015rate}
Nicolas Fournier and Arnaud Guillin.
\newblock On the rate of convergence in {W}asserstein distance of the empirical
  measure.
\newblock {\em Probability Theory and Related Fields}, 162(3):707--738, 2015.

\bibitem[FGS17]{FGS_book}
Giorgio Fabbri, Fausto Gozzi, and Andrzej \'{S}wi\c{e}ch.
\newblock {\em Stochastic optimal control in infinite dimension}, volume~82 of
  {\em Probability Theory and Stochastic Modelling}.
\newblock Springer, Cham, 2017.
\newblock Dynamic programming and HJB equations, With a contribution by Marco
  Fuhrman and Gianmario Tessitore.

\bibitem[FK09]{Feng_Kats_2009}
Jin Feng and Markos Katsoulakis.
\newblock A comparison principle for {H}amilton-{J}acobi equations related to
  controlled gradient flows in infinite dimensions.
\newblock {\em Arch. Ration. Mech. Anal.}, 192(2):275--310, 2009.

\bibitem[FMZ21]{Feng_Mikami_Zimmer_2021}
Jin Feng, Toshio Mikami, and Johannes Zimmer.
\newblock A {H}amilton-{J}acobi {PDE} associated with hydrodynamic fluctuations
  from a nonlinear diffusion equation.
\newblock {\em Comm. Math. Phys.}, 385(1):1--54, 2021.

\bibitem[FN12]{Feng2012}
Jin Feng and Truyen Nguyen.
\newblock {Hamilton-Jacobi equations in space of measures associated with a
  system of conservation laws}.
\newblock {\em Journal des Mathematiques Pures et Appliquees}, 97(4):318--390,
  2012.

\bibitem[FS89]{flemingsoug1989}
W.~H. Fleming and P.~E. Souganidis.
\newblock On the existence of value functions of two-player, zero-sum
  stochastic differential games.
\newblock {\em Indiana University Mathematics Journal}, 38(2):293--314, 1989.

\bibitem[FS13]{Feng_Swiech_2013}
Jin Feng and Andrzej \'{S}wi\c{e}ch.
\newblock Optimal control for a mixed flow of {H}amiltonian and gradient type
  in space of probability measures.
\newblock {\em Trans. Amer. Math. Soc.}, 365(8):3987--4039, 2013.
\newblock With an appendix by Atanas Stefanov.

\bibitem[GBR{\etalchar{+}}12]{Gretton_2012}
Arthur Gretton, Karsten~M. Borgwardt, Malte~J. Rasch, Bernhard Sch\"{o}lkopf,
  and Alexander Smola.
\newblock A kernel two-sample test.
\newblock {\em J. Mach. Learn. Res.}, 13:723--773, 2012.

\bibitem[GMS21]{GangboMayorgaSwiech}
Wilfrid Gangbo, Sergio Mayorga, and Andrzej Swiech.
\newblock Finite dimensional approximations of {H}amilton-{J}acobi-{B}ellman
  equations in spaces of probability measures.
\newblock {\em SIAM J. Math. Anal.}, 53(2):1320--1356, 2021.

\bibitem[GT19]{GT_19}
Wilfrid Gangbo and Adrian Tudorascu.
\newblock On differentiability in the {W}asserstein space and well-posedness
  for {H}amilton-{J}acobi equations.
\newblock {\em J. Math. Pures Appl. (9)}, 125:119--174, 2019.

\bibitem[Lac17]{Lacker2017}
Daniel Lacker.
\newblock {Limit theory for controlled mckean-vlasov dynamics}.
\newblock {\em SIAM Journal on Control and Optimization}, 55(3):1641--1672,
  2017.

\bibitem[Mou21]{mourrat}
Jean-Christophe Mourrat.
\newblock Nonconvex interactions in mean-field spin glasses.
\newblock {\em Probability and Mathematical Physics}, 2(281--339):437--461,
  2021.

\bibitem[MS23]{mayorgaswiech}
Sergio Mayorga and Andrzej \'{S}wi\c{e}ch.
\newblock Finite dimensional approximations of hamilton–jacobi–bellman
  equations for stochastic particle systems with common noise.
\newblock {\em SIAM Journal on Control and Optimization}, 61(2):820--851, 2023.

\bibitem[Sha23]{shao2023}
Jinghai Shao.
\newblock Viscosity solutions to hjb equations associated with optimal control
  problem for mckean-vlasov sdes.
\newblock 2023.

\bibitem[SY22]{Soner2022}
H.~Mete Soner and Qinxin Yan.
\newblock {Viscosity Solutions for McKean-Vlasov Control on a torus}.
\newblock {\em arxiv:2212.11053}, pages 1--21, 2022.

\bibitem[SY23]{Soner2023}
H.~Mete Soner and Qinxin Yan.
\newblock Viscosity solutions of the eikonal equation on the wasserstein space.
\newblock {\em arXiv preprint}, 2023.

\bibitem[TTZ23]{talbitouzizhang}
Mehdi Talbi, Nizar Touzi, and Jianfeng Zhang.
\newblock Viscosity solutions for obstacle problems on wasserstein space.
\newblock {\em SIAM Journal on Control and Optimization}, 61(3):1712--1736,
  2023.

\end{thebibliography}

\end{document}